\documentclass[reqno,twoside,12pt,a4paper]{amsart}
\usepackage{latexsym}
\usepackage{amsmath}
\usepackage{amssymb}
\usepackage{cases}
\usepackage{mathrsfs}
\usepackage{amsthm}
\usepackage{cite}
\def\interior{\mathop{\rm int}\nolimits}

\usepackage[usenames,dvipsnames]{color}

\def\pier #1{#1}
\def\takeshi #1{#1}
\def\pcol #1{#1}

\newtheorem{theorem}{\textbf{Theorem}}[section]
\newtheorem{lemma}{\textbf{Lemma}}[section]
\newtheorem{proposition}{\textbf{Proposition}}[section]

\numberwithin{equation}{section}
\newtheorem{lemmaa}{\textbf{Lemma A.\hskip-4pt}}
\newtheorem{remarka}{\textbf{Remark A.\hskip-4pt}}
\renewcommand{\theequation}{\thesection.\arabic{equation}}

\title[Vanishing \pier{boundary} diffusion in a {C}ahn--{H}illiard equation]
{Vanishing diffusion in a dynamic boundary condition\\[2mm] for the {C}ahn--{H}illiard equation}

\author[P. Colli]{Pierluigi Colli}
\address{Pierluigi Colli: 
Dipartimento di Matematica, Universit\`a degli Studi di Pavia\\
and Research Associate at the IMATI -- C.N.R. Pavia\\
Via Ferrata~5, 27100 Pavia, Italy}
\email{pierluigi.colli@unipv.it}

\author[T.\ Fukao]{Takeshi Fukao}
\address{Takeshi Fukao: Department of Mathematics, Faculty of Education, 
Kyoto University of Education, 
1~Fujinomori, Fukakusa, Fushimi-ku, Kyoto~612-8522 Japan}
\email{fukao@kyokyo-u.ac.jp}

\dedicatory{}
\subjclass[2000]{}
\begin{document}

\thispagestyle{empty}
\maketitle

\begin{abstract}
The initial boundary value problem for a 
{C}ahn--{H}illiard system subject to a dynamic boundary condition 
of {A}llen--{C}ahn type is \pier{treated. 
The vanishing} of the surface diffusion on the dynamic boundary condition is the point of emphasis. 
By the asymptotic \pier{analysis as the diffusion coefficient tends to $0$, 
one can expect that the solutions of the surface diffusion problem converge to 
the solution of the problem without the surface diffusion. This is actually the case, but
the solution of the limiting problem naturally looses some regularity. 
Indeed, the system we investigate is rather complicate due to the presence 
of nonlinear terms including general maximal monotone graphs 
both in the bulk and on the boundary. The two graphs are
related each to the other by a growth condition, with the boundary graph  that dominates the other one.  In general, 
at the asymptotic limit a weaker form of the boundary condition is obtained,  but in the case  when the
two graphs exhibit the same growth the boundary condition still holds almost everywhere.} \\[2mm]
\end{abstract}

\section{Introduction}
\setcounter{equation}{0}

In this paper,  we \pier{deal with the initial boundary value problem for a 
{C}ahn--{H}illiard system subject to a dynamic boundary condition 
of {A}llen--{C}ahn type. We study the limiting behavior of the system
as  the coefficient of the diffusive term in the boundary condition goes to $0$ 
and we investigate the limit problem.}

\pier{The {C}ahn--{H}illiard system \cite{CH58} is a celebrated model describing the spinodal decomposition by the simple framework 
of partial differential equations.  It is a phenomenological model 
that find its root in the work of J.~W.~Cahn~\cite{Ca61}, 
who studied the effects of interfacial energy on the stability 
of spinodal states in solid binary solutions, and this
took origin from the previous collaboration with J.~{\takeshi E.}~Hilliard~\cite{CH58}.
In the recent decades, a lot of research contributions has been devoted 
to Cahn--Hilliard and viscous Cahn--Hilliard~\cite{No88, NP} systems. 
An impressive number of related references
can be found in the literature and it would be worthless to report a list here. 
Let us simply refer to the recent review paper\cite{Mi17}.}

Let $0<T<\infty$, 
and $\Omega \subset \mathbb{R}^{d}$ ($d=2$ or $3$) be a bounded domain with smooth boundary 
$\Gamma := \partial \Omega$. Let 
$u, \mu: Q:=(0,T) \times \Omega \to \mathbb{R}$ be two \pier{unknowns, representing}
the order parameter and the chemical potential, respectively.
They \pier{have to satisfy}
\begin{align} 
	& \partial_t u - \Delta \mu =0 
	\quad \text{in }Q, 
	\label{equ1}
	\\
	&\mu = \tau \partial_t u-\Delta u + F'(u)
	\quad \text{in }Q, 
	\label{equ2}
\end{align}
where $\tau \ge 0$ is a constant \pier{coefficient (the viscous {C}ahn--{H}illiard system
corresponds to the case} $\tau >0$);  
$\partial_t$  denotes 
the time derivative; 
$\Delta$ denotes the {L}aplacian. 
The information on \pier{the potential $F$ and its derivative  $F'$} is given later. 
From the viewpoint \pier{of partial differential equations, in order 
to solve the system~(\ref{equ1})--(\ref{equ2})} we need some auxiliary conditions, namely 
the \pier{boundary and initial conditions}. 
Here we set up the following boundary conditions and 
initial condition: 
\begin{align} 
	& \partial_{\boldsymbol{\nu}} \mu =0 
	\quad  \text{on }\Sigma, 
	\label{equ3}
	\\
	& u_{|_{\Gamma}} =u_\Gamma 
	\quad  \text{on }\Sigma,
	\label{equ4}
	\\
	& u(0)=u_0 \quad \text{in } \Omega,
	\label{equ6a}
\end{align}
where $\partial_{\boldsymbol{\nu}}$ denotes 
the outward normal derivative on $\Gamma$; 
$u_{|_\Gamma}$ stands for the trace of $u$ on the boundary $\Gamma$. 
The point of emphasis is \pier{now} the boundary condition for $u$. 
It seems that (\ref{equ4}) \pier{is setting a non-homogeneous {D}irichlet boundary condition for $u$,
instead  $u_\Gamma: \Sigma:=(0,T) \times \Gamma \takeshi{\to \mathbb{R}}$ is also unknown and is required to satisfy the dynamic boundary condition}
\begin{align} 
	&\partial_t u_\Gamma+ \partial_{\boldsymbol{\nu}} u  - \kappa \Delta_\Gamma u_\Gamma 
	+F'_\Gamma (u_\Gamma)=0
	\quad  \text{on }\Sigma, 
	\label{equ5}
\end{align}
where $\kappa > 0$ is a \pier{coefficient, intended to tend to 0, and it is {\it the most 
important parameter in this paper}. Besides, 
$\Delta _{\Gamma }$ denotes the {L}aplace--{B}eltrami operator
on $\Gamma $ (see, e.g., \cite[Chapter~3]{Gri09}). We observe that the coupling of  
(\ref{equ1})--(\ref{equ2}) and \eqref{equ5} gives rise to a sort of transmission problem. 
We can also set the initial  condition for $u_\Gamma$, i.e., 
\begin{align} 
	& u_\Gamma(0)=u_{0\Gamma} \quad \text{on }\Gamma
	\label{equ6}
\end{align}
to complete the problem.}

The nonlinear terms $F'$ and $F_\Gamma'$ are usually referred 
as \pier{the derivatives} of the double-well potentials $F$ and $F_\Gamma$. 
Therefore, the problem \pier{(\ref{equ1})--(\ref{equ6}) yields} the 
{C}ahn--{H}illiard system when $\tau =0$ (resp. the viscous {C}ahn--{H}illiard 
system \pier{if $\tau >0$}) 
with the \pier{{N}eumann homogeneous boundary condition (\ref{equ3})}
for the chemical  potential $\mu$ and the dynamic boundary condition of
{A}llen--{C}ahn type (\ref{equ5}) for \pier{the trace $u_\Gamma$ of the 
order parameter $u$.}
Typical and physically significant examples \pier{for potentials like $F$ and $F_\Gamma$}
are the so-called classical regular potential, the logarithmic potential,
and the double obstacle potential, which are defined by
\begin{align*}
  & {F}_{\rm reg}(r) := \frac 14 \, (r^2-1)^2 \,,
  \quad r \in \mathbb{R}, 
  \\
  & {F}_{\rm log}(r) := (1+r)\ln (1+r)+(1-r)\ln (1-r)  - c_1 r^2 \,,
  \quad r \in (-1,1),
  \\
  & {F}_{\rm obs}(r) := c_2 (1-r^2) 
  \quad \text{if $|r|\leq1$}
  \quad \text{and} \quad
  {F}_{\rm obs}(r) := +\infty
  \quad \text{if $|r|>1$} \takeshi{,}
\end{align*}
where 
$c_1>1$ and $c_2>0$ are positive constants\pier{, fixed so} that 
${F}_{\rm log}$ and ${F}_{\rm obs}$ are nonconvex.
\pier{In this paper, we treat the nonlinear terms  $F'$ in \eqref{equ2} and $F_\Gamma'$ in \eqref{equ5} by
separating  them into two parts, i.e., by assuming that $F'=\beta+\pi$ and $F_\Gamma'=\beta_\Gamma+\pi_\Gamma$,
where $\beta$, $\beta_\Gamma$ are the monotone parts (derivatives or subdifferentials of the convex parts of $F$ and $F_\Gamma$), 
while $\pi$, $\pi_\Gamma$ play as the (smooth) anti-monotone parts.
For example, in the case of the classical regular potential, 
$F_{\rm reg}'=\beta_{\rm reg}+\pi_{\rm reg}$
is specified as the derivative of $F_{\rm reg}$, that is
\begin{equation*}
	F'_{\rm reg}(r)=r^3-r, \quad \hbox{with} \quad \beta_{\rm reg}(r):=r^3, \quad \pi_{\rm reg}(r):=-r.
\end{equation*}
On the other hand, in the case of the non-smooth double obstacle potential $F_{\rm obs}$,
$\beta_{\rm obs}$ is defined by the subdifferential of the indicator function of $[-1,1]$, so to have 
\begin{equation}
	F'_{\rm obs}(r)=\partial I_{[-1,1]}(r) - 2 c_2 r, \quad \beta_{\rm obs}(r):=\partial I_{[-1,1]}(r), \quad \pi_{\rm obs}(r):=-2 c_2r. 
	\label{pier21}
\end{equation}
We point out that, as a general rule, we can always use the subdifferentials for $\beta$, $\beta_\Gamma$,
and these subdifferentials reduce to the derivatives whenever a derivative exists. Please note that the subdifferentials 
may also be multivalued graphs, as it happens in \eqref{pier21}.
Thus, we generally interpret the equation in (\ref{equ2}) by
\begin{equation} 
	\mu = \tau \partial_t u-\Delta u +\xi + \pi(u), \quad 
	\xi \in \beta(u)
	\quad \text{in }Q
		\label{pier22}
\end{equation}
and rewrite the boundary condition  (\ref{equ5})  as
\begin{align} 
	&\partial_t u_\Gamma+ \partial_{\boldsymbol{\nu}} u  - \kappa \Delta_\Gamma u_\Gamma 
	+\xi_\Gamma + \pi_\Gamma (u_\Gamma)=0,
	\quad  \xi_\Gamma \in \beta_\Gamma(u_\Gamma)
	\quad \text{on }\Sigma.
	\label{pier23}
\end{align}
Of course, we can chose different potentials for $F$ and $F_\Gamma$, which may lead 
in particular to different graphs $\beta$ and $\beta_\Gamma$; about possible relations 
among them, according to a rather usual setting
(cf., e.g.,\pcol{\cite{CC13, CF15b,CF15, CFW, CGNS17, CGS14, CGS17,CGS18, LW19, Sca19}}) here it is supposed 
that $\beta_\Gamma$ dominates $\beta$ in \pcol{the sense of} assumption (A2) in Section~2.    
We note that in our framework  it is possible to choose similar or even equal graphs $\beta$ and $\beta_\Gamma$.}

\pier{Next, let us emphasize a property of the PDE system under consideration. Indeed,
we see from (\ref{equ1}), (\ref{equ3}), (\ref{equ6a}) that the following mass conservation holds:
\begin{equation*}
	\int_{\Omega} u(t) dx = \int_{\Omega} u_0 dx 
	\quad \text{for all } t \in [0,T]. 
\end{equation*}
Of course, in the analysis of the problem expressed by (\ref{equ1}), (\ref{equ3})--(\ref{equ6}), \eqref{pier22}--\eqref{pier23}
this property of mass conservation plays a role. Let us review some contribution 
to this class of problems. As some pioneering result for the {C}ahn--{H}illiard system 
with a dynamic boundary condition of heat equation type, 
the global existence and uniqueness of solutions was treated in \cite{RZ03} and  
the convergence to equilibrium was shown in \cite{WZ04}. 
Moreover, the {C}ahn--{H}illiard system
with a semi-linear equation as dynamic boundary condition (including the {A}llen--{C}ahn equation), 
was addressed from different viewpoints: the long time behaviour was studied in \cite{GM09, MZ05}, while 
other investigations treated the coupling with the heat equation in \cite{Gal07}, the problem with memory in \cite{CGG11}, 
some regularity results for the problem with a singular potential in \cite{CGS14, GMS09}, 
the boundary mass constraint in \cite{CF15b}, and so on.}

\pier{In this paper, we focus on the asymptotic analysis of the surface diffusion term on the dynamic boundary condition (\ref{pier23}). 
By the asymptotic limit as $\kappa \searrow 0$, one can wonder whether
the solution of the problem with surface diffusion converges to 
the one of the same problem without surface diffusion, i.e., with \eqref{pier23} replaced by 
\begin{align} 
	&\partial_t u_\Gamma+ \partial_{\boldsymbol{\nu}} u
	+\xi_\Gamma + \pi_\Gamma (u_\Gamma)=0,
	\quad  \xi_\Gamma \in \beta_\Gamma(u_\Gamma)
	\quad \text{on }\Sigma.
	\label{pier24}
\end{align}
The answer is in the affirmative, but it turns out that 
the solution of the limiting problem apparently looses regularity, due to the absence 
of the diffusive term on the boundary.  As the reader will see, in general the terms 
$\partial_{\boldsymbol{\nu}} u$ and $\xi_\Gamma$ in \eqref{pier24} are not functions but elements of a dual space,
and the inclusion $ \xi_\Gamma \in \beta_\Gamma(u_\Gamma)$ has to be suitably 
reinterpreted and generalized. However, in the case when the graphs $\beta$ and $\beta_\Gamma$
have the same growth, it is proven that the boundary condition \eqref{pier24} holds almost everywhere
on $\Sigma$.}

\pcol{In the light of our approach, we aim to quote the recent paper~\cite{Sca19}, which deals with
a highly nonlinear extension of the Cahn--Hilliard system with dynamic boundary condition, including nonlinear viscosity terms
in the equation corresponding to \eqref{pier22} and in the boundary condition. 
In fact, our asymptotic results can be compared with the ones contained in 
\cite[Theorem~2.11]{Sca19}, where a convergence statement similar to our Theorem~\ref{convergence} below is given 
in terms of a subsequence $\kappa_n$ going to $0$, both for a viscous Cahn--Hilliard type system 
and, under suitable regularity properties on data, for a nonlinear variation of the Cahn--Hilliard system. However, 
due to the presence of additional nonlinearities, nothing is investigated in~\cite{Sca19} about convergence to solutions satisfying the 
boundary condition almost everywhere (as instead we do here).}

A brief outline of the present paper along with a short description of the various items is as follows.
In Section~2, after \pier{setting up the notation and technical tools, as well as the known results for the case $\kappa >0$, 
we will state the main theorems. These theorems are the convergence-existence result 
and the continuous dependence with respect to data when $\kappa=0$,
the uniqueness of the solution being included in the second 
theorem. 
In Section~3, we work with the approxiate solutions and collect the 
uniform estimates, being able to check the existence of 
solutions for the problem with $\kappa >0$. 
In Section~4, we consider the limiting procedure as $\kappa \searrow 0$ and give  
the proof of the continuous dependence result.  
In Section~5, we deal with an improvement of the convergence-existence theorem,
including some regularity properties for the solution, under the stronger assumption (A2)$^\prime$.}
\smallskip

Let us include a detailed index of sections and subsections as follows:\pier{%
\begin{itemize}
 \item[1.] Introduction
 \item[2.] Notation and main results
  \begin{itemize}
  \item[2.1.] Well-posedness for $\kappa\in (0,1]$
  \item[2.2.] Asymptotic analysis as $\kappa \searrow 0$
  \item[2.3.] Convergence-existence theorem
   \item[2.4.] Continuous dependence on data
 \end{itemize}
 \item[3.] Approximate solutions
 \begin{itemize}
  \item[3.1.] Uniform estimates
  \item[3.2.] Proof of Proposition~\ref{prop}
 \end{itemize}
 \item[4.] Proof of main theorems
  \begin{itemize}
  \item[4.1.] Proof of Theorem~\ref{convergence}
  \item[4.2.] Proof of Theorem~\ref{contdep}
 \end{itemize}
  \item[5.] Improvement of the convergence-existence theorem
  \item[] Appendix
  \item[] Acknowledgements
\end{itemize}%
}

\section{Notation and main results}
\setcounter{equation}{0}

In this paper, we \pier{deal with the following spaces}
\begin{equation*}
	H:=L^2(\Omega ), 
	\quad 
	H_\Gamma :=L^2(\Gamma ), 
	\quad 
	V:=H^1(\Omega ), 
	\quad 
	V_\Gamma :=H^1(\Gamma), 
	\quad W_\Gamma:=H^{1/2}(\Gamma),
\end{equation*}
\pier{that are all {H}ilbert spaces with respect to} the usual norms and inner products, 
\pier{denoted} by  
$| \cdot |_{H}$ and $(\cdot,\cdot )_{H}$, and so on. 
Moreover, $V^*$ and $V_\Gamma^*$ stand for the dual spaces 
of $V$ and $V_\Gamma$, respectively. \pier{The notation
$\langle \cdot, \cdot \rangle_{V^*,V}$ is used for the duality pairing} between 
\pier{$V^*$ and $V$. It is understood that $H$ (resp. $H_\Gamma$) is embedded in  $V^*$ (resp. $V_\Gamma^*$)
in the ususal way, i.e.,  $\langle w, z \rangle_{V^*,V} = (w,z )_{H} $ for all $w\in H$ and $z\in V$
(resp. $\langle w_\Gamma, z_\Gamma \rangle_{V_\Gamma^*,V_\Gamma} = (w_\Gamma,z_\Gamma )_{H_\Gamma} $ for all $w_\Gamma\in H_\Gamma$ and $z_\Gamma\in V_\Gamma$).}
We also use $W_\Gamma^*$ for the dual space of $W_\Gamma$ \pier{and, 
in view of the actual identification between $H_\Gamma$ and $H_\Gamma^*$,
it turns out that $V_\Gamma^*$ and $ W_\Gamma^*$ are completely isomorph to  
$H^{-1}(\Gamma)$ and} $ H^{-1/2}(\Gamma)$, respectively (see, \cite[Theorem~7.6, p.~36]{LM72}). 

\smallskip

We start from the following {C}ahn--{H}illiard system with the dynamic boundary condition 
of {A}llen--{C}ahn type: for \pier{all $\tau \in [0,1] $ and  $\kappa \in (0,1]$ (the upper 
bounds are taken as $1$ for simplicity), one has to find a quintuplet $(u,\mu, \xi, u_\Gamma, \xi_\Gamma )$ such that}
\begin{align} 
	& \partial_t u - \Delta \mu =0 
	\quad \text{a.e.\ in }Q, 
	\label{chd1}
	\\
	&\mu = \tau \partial_t u-\Delta u + \xi + \pi (u)-g , \quad 
	\xi \in \beta (u)
	\quad \text{a.e.\ in }Q, 
	\label{chd2}
	\\
	& \partial_{\boldsymbol{\nu}} \mu =0 
	\quad  \text{a.e.\ on }\Sigma, 
	\label{chd3}
	\\
	& u_{|_{\Gamma}} =u_\Gamma 
	\quad  \text{a.e.\ on }\Sigma,
	\label{chd4}
	\\ 
	&\partial_t u_\Gamma+ \partial_{\boldsymbol{\nu}} u  - \kappa \Delta_\Gamma u_\Gamma 
	+ \xi _\Gamma + \pi_\Gamma(u_\Gamma)= g_\Gamma, \quad 
	\xi_\Gamma \in \beta_\Gamma(u_\Gamma)
	\quad  \text{a.e.\ on }\Sigma, 
	\label{chd5}
	\\
	& u(0)=u_0 \quad \text{a.e.\ in } \Omega, 
	\quad u_\Gamma(0)=u_{0\Gamma} \quad \text{a.e.\ on }\Gamma{\takeshi ,}
	\label{chd6}
\end{align}
where $g : Q \to \mathbb{R}$, $g_\Gamma : \Sigma \to \mathbb{R}$, 
$u_0:\Omega \to \mathbb{R}$, $u_{0\Gamma}:\Gamma \to \mathbb{R}$ are given functions. 
Hereafter we assume that:
\begin{enumerate}
\item[(A1)]  $\beta$, $\beta_\Gamma$ are maximal monotone graphs in $\mathbb{R} \times \mathbb{R}$, which
\pier{coincide with}
the subdifferentials $\beta =\partial \widehat{\beta}$, $\beta_\Gamma =\partial \widehat{\beta}_\Gamma$ of 
some proper, lower semicontinuous, and convex functions 
$\widehat{\beta}$, $\widehat{\beta}_\Gamma: \mathbb{R} \to [0,+\infty]$ \pier{such} that 
$\widehat{\beta}(0)=\widehat{\beta}_\Gamma(0)=0$, with the 
corresponding effective domains denoted by $D(\beta)$ and \pier{$D(\beta_\Gamma)$}, respectively; 
\item[(A2)] $D (\beta_\Gamma) \subseteq D(\beta)$ and there exist \pier{two constants 
$\varrho \geq 1$ and $c_0>0$} such that 
\begin{equation}
	\bigl| \beta^\circ(r) \bigr| \le \varrho\bigl| \beta ^\circ_\Gamma (r) \bigr|+c_0 \quad 
	\text{for all } r \in D(\beta_\Gamma),
	\label{cccond}
\end{equation}
\pier{where $\beta^\circ $ and $ \beta^\circ_\Gamma$ denote the minimal sections of $\beta $ and $\beta_\Gamma$, specified by 
$\beta^\circ(r):=\{ r^* \in \beta(r) : |r^*|=\min_{s \in \beta(r)} |s|\}$, $r\in D(\beta)$, for $\beta$;}
\item[(A3)] $\pi$, $\pi_\Gamma: \mathbb{R} \to \mathbb{R}$ are {L}ipschitz continuous functions with their {L}ipschitz constants $L$ and $L_\Gamma$, respectively;   
\item[(A4)] $u_0 \in V$, $u_{0\Gamma} \in V_\Gamma$ \pier{satisfy} 
$\widehat{\beta}(u_0) \in L^1(\Omega)$,
$\widehat{\beta}_\Gamma(u_{0\Gamma}) \in L^1(\Gamma)$,
 and $(u_0)_{|_\Gamma}=u_{0\Gamma}$ a.e.\ on $\Gamma$.  
Moreover\pier{, let}
\begin{equation*}
	m_0:=\frac{1}{|\Omega|} \int _{\Omega} u_0 dx \in \pier{\interior{}} D(\beta_\Gamma); 
\end{equation*}
\item[(A5)] $g \in L^2(0,T;H)$, $g_\Gamma \in L^2(0,T;H_\Gamma)$; \pier{in the case $\tau =0$ let} $g \in H^1(0,T;H)$ or $g \in L^2 (0,T;V)$.
\end{enumerate} 
As a remark, we \pier{deduce that $0 \in \beta(0)$ and $0 \in \beta_\Gamma(0)$ as consequences} 
from the assumption~(A1). \pier{Next, we set $|\Omega|:=\int_\Omega 1 dx$ 
and point out that the condition (\ref{cccond}) in (A2) is very useful in the treatment of two different potentials,
$\beta$ in the bulk $\Omega$ and $\beta_\Gamma$ on the boundary $\Gamma$ (to our knowledge, 
this condition has been used for the first time in \cite{CC13}). Moreover, let us note that in the case when $\beta$ and $\beta_\Gamma$ satisfy  (\ref{cccond}) for some $\takeshi{\varrho} \in (0,1)$, then (A2) also works for $\takeshi{\varrho}=1$.}

\pier{%
\subsection{Well-posedness for $\kappa\in (0,1]$}
The system (\ref{chd1})--(\ref{chd6}) and some extensions of 
it have been shown to be well posed \cite{CF15b, CGS14} provided that 
$\kappa>0$. More precisely, in view of the results proved in \cite{CF15b, CGS14} we can state the following proposition.
\smallskip
\begin{proposition} \label{prop} Under the assumptions {\rm (A1)--(A5)}, there exists a unique (weak) solution  
$(u_\kappa ,\mu_\kappa , \xi_\kappa , u_{\Gamma, \kappa }, \xi_{\Gamma, \kappa } )$ such that
\begin{align*}
	& u_\kappa  \in H^1(0,T;V^*)\cap L^\infty (0,T;V) \cap L^2\bigl( 0,T;H^2(\Omega) \bigr), \\
	& \tau u_\kappa  \in H^1(0,T;H) \cap C\bigl( [0,T]; V\bigr), \\
	& \mu_\kappa  \in L^2(0,T;V), \quad \xi \in L^2(0,T;H), \\
	& u_{\Gamma, \kappa } \in H^1(0,T;H_\Gamma )\cap C \bigl( [0,T];V_\Gamma \bigr) \cap L^2 \bigl( 0,T;H^2(\Gamma) \bigr), \\
	& \xi_{\Gamma, \kappa } \in L^2(0,T;H_\Gamma)
\end{align*}
and satisfying {\rm (\ref{chd1})--(\ref{chd6})}; if $\tau=0$, then equation~{\rm (\ref{chd1})} and boundary condition~{\rm (\ref{chd3})} make sense in terms of the variational formulation 
\begin{equation}
\label{pier11}
	\bigl \langle \partial_t u(t), z \bigr \rangle_{V^*,V} + \int_\Omega \nabla \mu(t)\cdot \nabla z dx =0
	\quad \text{ for all } z \in V, \hbox{ for a.a.\ $t \in (0,T)$}.   
\end{equation}
\end{proposition}%
}

\pier{%
\subsection{Asymptotic analysis as $\kappa\searrow 0$}
We aim to discuss the asymptotic analysis of 
\pier{the system~(\ref{chd1})--(\ref{chd6})} as the coefficient $\kappa $ of the term 
with the Laplace--Beltrami operator in \eqref{chd5} tends to $0$.}
Then, we will obtain \pier{the singular limit problem, with the equality in (\ref{chd5})  replaced by 
\begin{equation}
	\partial_t u_\Gamma+ \partial_{\boldsymbol{\nu}} u 
	+ \xi _\Gamma + \pi_\Gamma(u_\Gamma)= g_\Gamma ,
	\label{pier1}
\end{equation}
meant in some dual space, and the characterization of the inclusion in (\ref{chd5}), i.e., 
\begin{equation}
	\xi_\Gamma \in \beta_\Gamma(u_\Gamma) \quad \text{a.e.\ on }\Sigma,
	\label{relation}
\end{equation}
in a weaker form as well.}
For this purpose, 
we \pier{introduce the maximal monotone operator 
$\beta_{\Gamma, W_\Gamma^*}:W_\Gamma \to 2^{W_\Gamma^*}$, 
which is the subdifferential of the functional
\pier{$\widehat{\beta}_{\Gamma, W_\Gamma}:W_\Gamma \to [0,+\infty]$ defined below (cf., e.g.,} \cite[Section~5]{CC13}). In fact, we set}
\begin{align*}
	& \widehat{\beta}_{\Gamma,H_\Gamma}(z_\Gamma):=
	\begin{cases}
	\displaystyle \int_\Gamma \widehat{\beta}_\Gamma(z_\Gamma) 
	& \text{if } z \in H_\Gamma \text{ and } \widehat{\beta}_\Gamma(z_\Gamma) \in L^1(\Gamma), \\
	\pier{{}+\infty} & \text{if } z \in H_\Gamma \text{ and } \widehat{\beta}_\Gamma(z_\Gamma) \notin L^1(\Gamma), 
	\end{cases}
\end{align*}
\pier{and, subsequently, 	
\begin{equation*}
	\widehat{\beta}_{\Gamma,W_\Gamma}(z_\Gamma):=\widehat{\beta}_{\Gamma,H_\Gamma}(z_\Gamma) 
	\quad  \text{if } \pier{z_\Gamma} \in W_\Gamma. 
\end{equation*}
We notice that both functionals $\widehat{\beta}_{\Gamma,H_\Gamma}$ and $ \widehat{\beta}_{\Gamma,W_\Gamma} $ 
are proper (equal to $0$ for $ z_\Gamma=0$),} lower semicontinous, and convex on $H_\Gamma$ and $W_\Gamma$, respectively. Now, \pier{if we set 
\begin{equation*}
	\beta_{\Gamma,H_\Gamma}
	:=\partial \widehat{\beta}_{\Gamma,H_\Gamma}:H_\Gamma \to 2^{H_\Gamma},
\end{equation*}
then we have that
$z_\Gamma^* \in \beta_{\Gamma,\takeshi{H_\Gamma}}(z_\Gamma)$ in $H_\Gamma$ if and only if 
$z_\Gamma^* \in H_\Gamma$, $z_\Gamma \in D(\widehat{\beta}_{\Gamma,H_\Gamma})$, and
\begin{gather*}
	( z_\Gamma^*, \tilde{z}_\Gamma- z_\Gamma )_{H_\Gamma} \le 
	\widehat{\beta}_{\Gamma,H_\Gamma}(\tilde{z}_\Gamma)-
	\widehat{\beta}_{\Gamma,H_\Gamma}(z_\Gamma)
	\quad \text{for all } \tilde{z}_\Gamma \in H_\Gamma.
\end{gather*}
We emphasize that $\beta_{\Gamma,H_\Gamma}$ is nothing but the operator induced by $\beta_\Gamma$ on $H_\Gamma$,
so that (see, e.g., \cite{Bre73}) the inclusion (\ref{relation}) can be equivalently rewritten as 
\begin{equation*}
	\xi_\Gamma (t)  \in \beta_{\Gamma,H_\Gamma} \takeshi{\bigl (}u_\Gamma (t) \takeshi{\bigr)} 
	\quad \text{in } H_\Gamma, \ \hbox{ for a.a. } t \in (0,T).
\end{equation*}
On the other hand, for
\begin{equation*}
	\beta_{\Gamma,W_\Gamma^*}
	:=\partial \widehat{\beta}_{\Gamma,W_\Gamma}:W_\Gamma \to 2^{W_\Gamma^*},
\end{equation*}
we remark that
$z_\Gamma^* \in \beta_{\Gamma,W_\Gamma^*}(z_\Gamma)$ in $W_\Gamma^*$ if and only if 
$z_\Gamma^* \in W_\Gamma^*$, $z_\Gamma \in D(\widehat{\beta}_{\Gamma,W_\Gamma})$, and
\begin{gather*}
	\langle z_\Gamma^*, \tilde{z}_\Gamma- z_\Gamma \rangle_{W_\Gamma^*, W_\Gamma} \le 
	\widehat{\beta}_{\Gamma,W_\Gamma}(\tilde{z}_\Gamma)-
	\widehat{\beta}_{\Gamma,W_\Gamma}(z_\Gamma)
	\quad \text{for all } \tilde{z}_\Gamma \in W_\Gamma.
\end{gather*}
Hence, it is clear that 
%
$	z^*_\Gamma \in \beta_{\Gamma,H_\Gamma} (z_\Gamma ) 
	\  \text{in } H_\Gamma $
%
entails 
$	z_\Gamma^* \in \beta_{\Gamma,W_\Gamma^*} 
	(z_\Gamma) \  \text{in } W_\Gamma^*$
%
whenever $z_\Gamma^* \in H_\Gamma \subset W_\Gamma^*$ and 
$z_\Gamma \in D(\widehat{\beta}_{\Gamma,W_\Gamma}) \subset W_\Gamma \subset H_\Gamma$. Then 
a possible extension of \eqref{relation} is
\begin{equation*}
	\xi_\Gamma (t)  \in \beta_{\Gamma,W^*_\Gamma} \takeshi{ \bigl(}u_\Gamma (t) \takeshi{\bigr )}
	\quad \text{in } W^*_\Gamma, \ \hbox{ for a.a. } t \in (0,T),
	\label{pier2}
\end{equation*}
in the case when $\xi_\Gamma (t) \not\in H_\Gamma $ for all $t$ of a set which is not negligible.}

\pier{%
\subsection{Convergence-existence theorem}
Our main result is stated here.}
\begin{theorem}
\label{convergence}
 Let $\tau \ge 0$ \pier{and assume that {\rm (A1)--(A5)} hold. For all $\kappa\in (0,1]$
let 
$(u_\kappa ,\mu_\kappa , \xi_\kappa , u_{\Gamma, \kappa }, \xi_{\Gamma, \kappa } )$ denote the solution to 
\eqref{chd1}--\eqref{chd6} defined by Proposition~\ref{prop}. Then
exists  a quintuplet $(u,\mu, \xi, u_\Gamma, \xi_\Gamma )$ such that
\begin{align}
	u_\kappa \to u 
	& \quad \text{weakly star in } H^1(0,T;V^*) \cap L^\infty(0,T;V), 
	\nonumber \\
	& \quad \text{and strongly in } C\bigl( [0,T];H\bigr), 
	\label{u1}\\
	\tau u_\kappa \to \tau u 
	& \quad \text{weakly in } H^1(0,T;H), 
	\label{pier4}\\
		\mu_\kappa \to \mu 
	& \quad \text{weakly in } L^2(0,T;V), 
	\label{mu1}\\
	\xi_\kappa \to \xi 
	& \quad \text{weakly in } L^2(0,T;H), 
	\label{beta1}\\
	\Delta u_\kappa \to \Delta u
	& \quad \text{weakly in } L^2(0,T;H), 
	\label{Lap1}\\
	u_{\Gamma,\kappa} \to u_{\Gamma}
	&  \quad \text{weakly star in } H^1(0,T;H_\Gamma) \cap L^\infty ( 0,T;W_\Gamma), 
	\nonumber \\
	& \quad \text{and strongly in } C\bigl( [0,T];H_\Gamma \bigr), 
	\label{ug1}\\
	\partial_{\boldsymbol{\nu}} u_\kappa \to \partial_{\boldsymbol{\nu}} u 
	\label{nu1}
	& \quad \text{weakly in } L^2 (0,T;W_\Gamma^*), \\
	\xi_{\Gamma,\kappa} \to \xi_\Gamma
	&  \quad \text{weakly in } L^2 (0,T;W_\Gamma^*) 
	\label{xig1}
\end{align}
as $\kappa \searrow 0$. \takeshi{Moreover}, the limit functions  $ u,\mu, \xi, u_\Gamma, \xi_\Gamma  $ satisfy}
%
\begin{gather}
	\bigl \langle \partial_t u(t), z \bigr \rangle_{V^*,V} 
	+ \int_\Omega \nabla \mu(t)\cdot \nabla z dx =0
	\quad \text{ for all } z \in V, 
	\  \text{for a.a.\ } t \in (0,T),
	\label{pier5}\\
	\mu = \tau \partial_t u -\Delta u 
	+ \xi + \pi (u)-g, \quad \xi \in \beta (u)
	\quad \text{a.e.\ in }Q, 
	\label{pier6}\\[2mm]
	u_{|_{\Gamma}} =u_\Gamma 
	\quad  \text{a.e.\ on }\Sigma,
	\label{pier7}\\
	\int_\Gamma \partial_t u_\Gamma(t) z_\Gamma d\Gamma 
	+ \bigl \langle \partial_{\boldsymbol{\nu}} u(t), z_\Gamma \bigr \rangle_{W_\Gamma^*, W_\Gamma} 
	+  \bigl \langle \xi_\Gamma (t), z_\Gamma \bigr \rangle_{W_\Gamma^*, W_\Gamma} 
	+ \int_\Gamma \pi_\Gamma \bigl(u_\Gamma (t) \bigr) z_\Gamma d\Gamma 
	\nonumber \\
	{}= \int_\Gamma g_\Gamma (t) z_\Gamma d \Gamma 
	\quad \text{ for all } z_\Gamma \in \pier{W_\Gamma}, 
	\  \text{for a.a.\ } t \in (0,T), 
	\label{pier8}\\
	\xi_\Gamma (t) \in \beta_{\Gamma,W_\Gamma^*} \bigl( u_\Gamma (t)\bigr) 
	\quad \text{in }W_\Gamma ^*, \ \,\text{for a.a.\ } t \in (0,T), \label{pier9}\\
	u(0)=u_0 \quad \text{a.e.\ in } \Omega, 
	\quad u_\Gamma(0)=u_{0\Gamma} \quad \text{a.e.\ on }\Gamma. \label{pier10}
\end{gather}
\end{theorem}

\smallskip

\pier{The above result is simultaneously a {\it convergence} theorem of the solutions to 
(\ref{chd1})--(\ref{chd6}) toward the solution of the limiting problem, which then 
{\it exists} and  turns out to be unique, as the next result will confirm.}

\pier{Note that the equation \eqref{pier1} on the boundary is expressed in the weak form
\eqref{pier8}. We remark that the 
regularity of the solution to \eqref{pier5}--\eqref{pier10} is not enough to conclude that 
$ \xi_\Gamma$ belong to $ L^2 (0,T;H_\Gamma)$
(and in this case \eqref{pier1} may hold a.e.~on $\Sigma$), although $\Delta u \in L^2(0,T;H)$.
Indeed, if we recall the elliptic regularity theorem \cite[Theorem~3.2, p.~1.79]{BG87}
for $u(t) \in V$ solving 
\begin{equation*}
\begin{cases}
	-\Delta u (t)= \tilde g  (t)\quad \text{a.e.\ in } \Omega, 
	\\
	u_{|_\Gamma} (t) =u_\Gamma (t) \quad \text{a.e.\ on } \Gamma,
\end{cases}
\end{equation*}
for a.a. $t\in (0,T)$, where $\tilde g=\mu-\tau \partial_t u -\xi -\pi(u)+g \in L^2(0,T;H) $ and 
$u_{\Gamma} \in L^2(0,T;W_\Gamma) $, we can just infer that $ \partial_{\boldsymbol{\nu}} u \in
L^2 (0,T;W_\Gamma^*)$. In this regard, we announce that an improvement of this theorem is 
given in Section~5, under a special assumption on the graphs $\beta $ and $\beta_\Gamma$.}
\smallskip

\pier{%
\subsection{Continuous dependence on data}
For the solutions to the problem~\eqref{pier5}--\eqref{pier10} we can 
prove a continuous dependence result, which in particular implies uniqueness.}

\begin{theorem} 
\label{contdep}
Take two sets of data $u_{0,i}, u_{0\Gamma,i}, g_i, g_{\Gamma,i}$ for $i=1,2$ satisfying {\rm (A4)--(A5)} \pier{and
\begin{equation}
	\frac{1}{|\Omega|} \int_\Omega u_{0,1} dx = \frac{1}{|\Omega|} \int_\Omega u_{0,2} dx =m_0.
	\label{media}
\end{equation}
Let $(u_i,\mu_i, \xi_i, u_{\Gamma,i},  \xi_{\Gamma,i})$, $i=1,2$, denote the corresponding solutions to \eqref{pier5}--\eqref{pier10}.}
Then there exists a positive constant $C$, depending only on $L$, $L_\Gamma$, and $T$, such that 
\begin{align}
	& |u_1-u_2|_{C([0,T;V^*)}^2 
	+ \tau |u_1-u_2|_{C([0,T];H)}^2
	+ |u_1-u_2|_{L^2(0,T;V)}^2 
	+ |u_{\Gamma,1}-u_{\Gamma,2}|_{C([0,T];H_\Gamma)}^2
	\nonumber \\
	& \le C \Bigl\{ 
	 |u_{0,1}-u_{0,2}|_{V^*}^2 
	+ \tau |u_{0,1}-u_{0,2}|_{H}^2
	+ |u_{0\Gamma,1}-u_{0\Gamma,2}|_{H_\Gamma}^2
	\nonumber \\
	& \quad {} + 
	|g_1-g_2|_{L^2(0,T;H)}^2 
	+ 
	|g_{\Gamma,1}-g_{\Gamma,2}|_{L^2(0,T;H_\Gamma)}^2 
	\Bigr\}.
	\label{dep}
\end{align}
\end{theorem}

\smallskip

Of course, this theorem implies the uniqueness of the solution \pier{$(u,\mu, \xi, u_\Gamma, \xi_\Gamma )$  obtained
by the limit procedure in Theorem~\ref{convergence}. The theorem will be proved in Section~4.}

\pier{%
\section{Approximate solutions}
\setcounter{equation}{0}
In this section we sketch the main steps for proving the existence of a solution stated in Proposition~\ref{prop}
and, at the same time, we will derive uniform estimates that will be useful in the proof of Theorem~\ref{convergence}.}

\pier{Thus, we approximate the problem \eqref{chd1}--\eqref{chd6} by introducing the {Y}osida  regularizations $\beta_\varepsilon$ for 
$\beta$ and   $\beta _{\Gamma,\varepsilon }$ for $\beta _{\Gamma} $ (see, e.g., \cite{Bar10, Bre73}): 
for each $\varepsilon  \in (0,1]$ and for all $r \in \mathbb{R}$ we set}
\begin{align}
	&\beta _\varepsilon (r)
	:= \frac{1}{\varepsilon } \bigl( r-J_\varepsilon (r) \bigr), 
	\quad 
	J_\varepsilon (r) 
	:=(I+\varepsilon  \beta )^{-1} (r),
	\label{yosida1}
	\\
	&\beta _{\Gamma, \varepsilon } (r)
	:= \frac{1}{\varepsilon} \bigl( r-J_{\Gamma,\varepsilon }(r) \bigr ), 
	\quad 
	J_{\Gamma,\varepsilon }(r):=(I+\varepsilon \beta _\Gamma )^{-1} (r).
	\label{yosida2}
\end{align}
Then, we \pier{point out that the condition (\ref{cccond}) in (A2) implies that
\begin{equation}
	\bigl |\beta_\varepsilon (r)\bigr | 
	\le \varrho \bigl |\beta _{\Gamma,\varepsilon } (r)\bigr |+c_0
	\quad 
	\text{for all } r \in \mathbb{R},
	\label{ccconde}
\end{equation} 
for all $\varepsilon \in (0,1]$, with the same constants $\varrho $ and $c_0$ (see Appendix)}.  
We also have 
$\beta _\varepsilon (0)=\beta _{\Gamma, \varepsilon }(0)=0$. 
\smallskip

\pier{Then, the problem in terms of the $\varepsilon$-approximation reads as follows: find a triplet $(u_\varepsilon, \mu_\varepsilon, u_{\Gamma,\varepsilon})$, with at least the same regularity as $(u_\kappa, \mu_\kappa, u_{\Gamma,\kappa})$ in Proposition~\ref{prop}, satisfying
\begin{align} 
	& \partial_t u_\varepsilon - \Delta \mu_\varepsilon =0 
	\quad \text{a.e.\ in }Q, 
	\label{chd1e}
	\\
	&\mu_\varepsilon = \tau \partial_t u_\varepsilon-\Delta u_\varepsilon + \beta_\varepsilon(u_\varepsilon) + \pi (u_\varepsilon)-g
	\quad \text{a.e.\ in }Q, 
	\label{chd2e}
	\\
	& \partial_{\boldsymbol{\nu}} \mu_\varepsilon =0 
	\quad  \text{a.e.\ on }\Sigma, 
	\label{chd3e}
	\\
	& (u_{\varepsilon})_{|_{\Gamma}} =u_{\Gamma,\varepsilon} 
	\quad  \text{a.e.\ on }\Sigma,
	\label{chd4e}
	\\ 
	&\partial_t u_{\Gamma,\varepsilon}+ \partial_{\boldsymbol{\nu}} u_\varepsilon  
	- \kappa \Delta_\Gamma u_{\Gamma,\varepsilon} 
	+ \beta_{\Gamma, \varepsilon} (u_{\Gamma,\varepsilon}) 
	+ \pi_\Gamma(u_{\Gamma, \varepsilon})= g_\Gamma
	\quad  \text{a.e.\ on }\Sigma, 
	\label{chd5e}
	\\
	& u_\varepsilon(0)=u_0 \quad \text{a.e.\ in } \Omega, 
	\quad u_{\Gamma,\varepsilon}(0)=u_{0\Gamma} \quad \text{a.e.\ on }\Gamma,
	\label{chd6e}
\end{align}
where \eqref{chd1e} and \eqref{chd3e} have to be collected into the proper 
variational formulation if $\tau= 0$ (cf.~\eqref{pier11}). From the results shown in~\cite{CF15b, CGS14} it follows that
there exists such a triplet~$(u_\varepsilon, \mu_\varepsilon, u_{\Gamma,\varepsilon})$ and, in addition, it is unique.
We observe that for the proof one can use the abstract theory of doubly nonlinear evolution equations presented in~\cite{CV90} 
and argue in the function spaces} 
\begin{align*}
	& H_0:=\bigl\{ x \in H \ : \ m(z)=0 \bigr\}, \\
	& V_0:=V \cap H_0, \\
	& V_{0*}:=\bigl\{ z^* \in V^* \ : \ \langle z^*,1 \rangle _{V^*,V} =0 \bigr\},
\end{align*}
where the mean value function $m:H \to \mathbb{R}$ is defined by 
\begin{equation*}
	m(z):=\frac{1}{|\Omega|} \int _\Omega z dx.
\end{equation*}
As a remark, we can identify $V_{0*}$ by $V_0^*$ (see, \cite[Remark~2]{CF15}). 
Then, from the Poincar\'e inequality, we have that there exists a positive constant 
$C_{\rm P}$ such that 
\begin{equation}
	|z|_{V} \le C_{\rm P} |\nabla z|_{H} 
	\quad \text{for all } \pier{z \in V_0},
\label{pier3}
\end{equation}
 that is, we see that $| \cdot |_{V_0}:=|\nabla \cdot |_{H}$ and 
the standard $|\cdot |_{V}$ are equivalent \pier{norms} on $V_0$. 
Moreover, we introduce the linear, bijective, and symmetric operator ${\mathcal N}: V_{0*} \to V_0$ by 
$v={\mathcal N}v^* $ if and only if $m(v)=0$ and 
\begin{equation*}
	\int_{\Omega} \nabla {\mathcal N}v^* \cdot \nabla z dx =\int_{\Omega} \nabla v \cdot \nabla z dx = \langle v^*,z \rangle_{V^*,V} \quad \text{for all } z \in V.  
\end{equation*}
Then we can introduce the norm 
\begin{equation*}
	|z^*|_{V_{0\ast}}
	:= \left( \int_\Omega | \nabla {\mathcal N} z^*|^2 dx \right)^{\!\!1/2} 
	= |{\mathcal N} z^*|_{V_0}
	\quad \text{for all } z^* \in V_{0\ast}.
\end{equation*}
%

\pier{Next, in the light of \eqref{chd1e}--\eqref{chd6e} we prove and collect some estimates for $(u_\varepsilon, 
\mu_\varepsilon, u_{\Gamma,\varepsilon})$ that are independent of 
$\varepsilon$ and $\kappa \in (0,1]$. The dependence on $\tau$ will be explicitly mentioned when needed.}

\subsection{Uniform estimates}

\begin{lemma} Let $\tau \in (0,1]$\pier{; then} 
there exists a positive constant $M_1:=M_1(\tau)$ 
independent of \pier{$\varepsilon, \kappa \in (0,1]$} such that
\begin{gather}
	| u_\varepsilon |_{H^1(0,T;V^*) }
	+
	| u_\varepsilon |_{L^\infty(0,T;V)}
	+
	\sqrt{\tau} | \partial_t u_\varepsilon |_{L^2(0,T;H)}
	+ 
	| u_{\Gamma, \varepsilon} |_{H^1(0,T;H_\Gamma)}
	+
	| u_{\Gamma, \varepsilon} |_{L^\infty(0,T;W_\Gamma)}
	\nonumber \\
	{} +
	\sqrt{\kappa}
	| u_{\Gamma, \varepsilon} |_{L^\infty(0,T;V_\Gamma)}
	+ 
	\bigl| \widehat{\beta}_\varepsilon ( u_\varepsilon ) \bigr|_{L^\infty (0,T;L^1(\Omega))}
	+
	\bigl| \widehat{\beta}_{\Gamma, \varepsilon}( u_{\Gamma, \varepsilon} ) \bigr|_{L^\infty (0,T;L^1(\Gamma))}
	\le M_1.
	\label{lemma1}
\end{gather}
\pier{Otherwise, if $\tau=0$,} then there exists a positive constant $M_2$ 
independent of $\varepsilon, \kappa \in (0,1]$ such~that
\begin{gather}
	| u_\varepsilon |_{H^1(0,T;V^*) }
	+
	| u_\varepsilon |_{L^\infty(0,T;V)}
	+
	| u_{\Gamma, \varepsilon} |_{H^1(0,T;H_\Gamma)}
	+
	| u_{\Gamma, \varepsilon} |_{L^\infty(0,T;W_\Gamma)}
	\nonumber \\
	{} +
	\sqrt{\kappa}
	| u_{\Gamma, \varepsilon} |_{L^\infty(0,T;V_\Gamma)}
	+ 
	\bigl| \widehat{\beta}_\varepsilon ( u_\varepsilon ) \bigr|_{L^\infty (0,T;L^1(\Omega))}
	+
	\bigl| \widehat{\beta}_{\Gamma, \varepsilon}( u_{\Gamma, \varepsilon} ) \bigr|_{L^\infty (0,T;L^1(\Gamma))}
	\le M_2.
	\label{lemma1b}
\end{gather}
\end{lemma}

\begin{proof} 
\pier{By integrating (\ref{chd1e}) over $\Omega$ and using (\ref{chd3e}), we obtain $\int_\Omega
 \partial_t u_\varepsilon =0 $ a.e. in $(0,T)$. Hence, by integrating with respect to time and taking (\ref{chd6e})
 and (A2) into account, we easily have} 
\begin{equation}
	\frac{1}{|\Omega|} \int_\Omega u_\varepsilon(t) dx = \frac{1}{|\Omega|} \int_\Omega u_0 dx =m_0
	\label{masse}
\end{equation}
for all $t \in [0,T]$. \pier{Thus, we see that $ m(u_\varepsilon(t))=m_0 $   
and $\partial_t m(u_\varepsilon(s))=m(\partial_t u_\varepsilon(s))=0$, that is, 
$\partial_t u_\varepsilon(s) \in V_0$, for all $t \in [0,T]$ and  a.a.\ $s \in (0,T)$.}
Now, multiplying (\ref{chd1e}) by ${\mathcal N}(\partial _t u_\varepsilon)$, 
integrating the resultant over $\Omega$, and using (\ref{chd3e}) we obtain 
\begin{equation}
	\bigl( \partial_t u_\varepsilon(s), {\mathcal N} (\partial_t u_\varepsilon)(s) \bigr )_{\! H_0}
	+
	\int_\Omega \nabla \mu_\varepsilon(s) \cdot \nabla {\mathcal N} (\partial_t u_\varepsilon)(s) dx 
	=0
	\label{equu}
\end{equation}
for a.a.\ $s \in (0,T)$. \pier{The continuation of the estimate is formal, at least in the case $\tau=0$; 
however the reader may refer to 
\cite{CF15b, CGS14} for the details of a rigorous proof. Hence, adding $u_\varepsilon$ to both sides of 
\pier{(\ref{chd2e})}, testing it by $\partial_t u_\varepsilon$}  and using (\ref{chd5e}) we obtain 
\begin{gather}
	\int_\Omega \mu_\varepsilon(s) \partial_t u_\varepsilon(s) dx = 
	\tau \bigl| \partial_t u_\varepsilon(s)\bigr|_H^2+ 
	\frac{1}{2} \frac{d}{dt}\bigl| u_\varepsilon (s) \bigr|_V^2 
	+ \frac{d}{dt} \int_\Omega \widehat{\beta}_\varepsilon \bigl( u_\varepsilon(s) \bigr) dx 
	\nonumber \\
	{} + 
	\bigl| \partial_t u_{\Gamma, \varepsilon} (s) \bigr|_{H_\Gamma}^2 
	+ 
	\frac{\kappa}{2} \frac{d}{dt} \int_\Gamma \bigl| \nabla _\Gamma \pier{u_{\Gamma, \varepsilon}}(s) \bigr|^2 d\Gamma
	+
	\frac{d}{dt} \int_\Gamma \widehat{\beta}_{\Gamma, \varepsilon}\bigl( u_{\Gamma, \varepsilon} (s)\bigr) d\Gamma
	\nonumber \\
	{} - 
	\int_\Omega \bigl\{ g(s)+u_\varepsilon(s) -\pi\bigl( u_\varepsilon(s) \bigr) \bigr\} 
	\partial_t u_\varepsilon(s) dx 
	- 
	\int_\Gamma  \bigl\{ g_\Gamma (s) -\pier{\pi_\Gamma} \bigl( u_{\Gamma, \varepsilon}(s) \bigr) \bigr\} 
	\partial_t u_{\Gamma, \varepsilon}(s) d\Gamma
	\label{equmu}
\end{gather}
for a.a.\ $s \in (0,T)$. Here, \pier{let us remark} that 
\begin{align*}
	( z, {\mathcal N} z )_{H_0} 
	& = \langle z, {\mathcal N} z \rangle_{V^*,V} 
	 =  \int_\Omega \nabla {\mathcal N} z \cdot \nabla {\mathcal N} z dx \\
	& = |{\mathcal N} z |_{V_0}^2
	 = |z|_{V_{0*}}^2
	 = |z|_{V^*}^2 \quad \text{for all } z \in H_0,
\end{align*}
and 
\begin{align*}
	\int_\Omega \nabla z_1 \cdot \nabla ({\mathcal N} z_2) dx 
	 = \langle z_2, z_1 \rangle_{V^*,V} 
	 =  (z_2, z_1)_{\pier H} \quad \text{for all } \pier{z_1\in V, \ z_2 \in H_0}.
\end{align*}
Therefore, by  \pier{subtracting  (\ref{equmu}) from (\ref{equu}) we cancel two terms and obtain}
\begin{align}
	& \bigl| \partial_t u_\varepsilon(s)\bigr|_{V^*}^2 +
	\tau \bigl| \partial_t u_\varepsilon(s)\bigr|_H^2 + 
	\frac{1}{2} \frac{d}{dt}\bigl| u_\varepsilon (s) \bigr|_V^2 
	+ \frac{d}{dt} \int_\Omega \widehat{\beta}_\varepsilon \bigl( u_\varepsilon(s) \bigr) dx 
	\nonumber \\
	& \quad {} + 
	\bigl| \partial_t u_{\Gamma, \varepsilon} (s) \bigr|_{H_\Gamma}^2 
	+ 
	\frac{\kappa}{2} \frac{d}{dt} \int_\Gamma \bigl| \nabla _\Gamma \pier{u_{\Gamma, \varepsilon}} (s) \bigr|^2 d\Gamma
	+
	\frac{d}{dt} \int_\Gamma \widehat{\beta}_{\Gamma, \varepsilon}\bigl( u_{\Gamma, \varepsilon} (s)\bigr) d\Gamma
	\nonumber \\
	& \le 
	\int_\Omega \bigl\{ g(s)+u_\varepsilon(s) -\pi\bigl( u_\varepsilon(s) \bigr) \bigr\} 
	\partial_t u_\varepsilon(s) dx 
	+
	\int_\Gamma  \bigl\{ g_\Gamma (s) -\pier{\pi_\Gamma}\bigl( u_{\Gamma, \varepsilon}(s) \bigr) \bigr\} 
	\partial_t u_{\Gamma, \varepsilon}(s) d\Gamma
	\label{ineumu}
\end{align}
for a.a.\ $s \in (0,T)$. Then, we remark that there exists some positive constant $C_1$, depending on $L_\Gamma$, such that 
\begin{align}
	& \left| \int_\Gamma  \bigl\{ g_\Gamma (s) -\pier{\pi_\Gamma}\bigl( u_{\Gamma, \varepsilon}(s) \bigr) \bigr\} 
	\partial_t u_{\Gamma, \varepsilon}(s) d\Gamma \right| \nonumber \\
	& \quad  \le \frac{1}{2} \bigl| \partial_t u_{\Gamma, \varepsilon}(s) \bigr|_{H_\Gamma}^2 
	+ \bigl| g_\Gamma (s) \bigr|_{H_\Gamma}^2 + C_1\bigl( 1+ \bigl|  u_{\Gamma,\varepsilon}(s) \bigr|_{H_\Gamma}^2 \bigr)
	\label{estong}
\end{align}
for a.a.\ $s \in (0,T)$. 
In the case when $\tau>0$, \pier{the first term on the right-hand side of \eqref{ineumu} can be handled with the help of the Young inequality, as}
\begin{align}
	& \left| \int_\Omega \bigl\{ g(s)+u_\varepsilon(s) -\pi\bigl( u_\varepsilon(s) \bigr) \bigr\} 
	\partial_t u_\varepsilon(s) dx  \right| \nonumber \\
	&  \quad \le \frac{\tau}{2} \bigl| \partial_t u_\varepsilon(s) \bigr|_{H_\Gamma}^2 
	+ \frac{3}{2\tau} \bigl| g(s) \bigr|_{H}^2 + \frac{C_2}{\tau} \bigl( 1+ \bigl|  u_\varepsilon(s) \bigr|_{H}^2 \bigr),
	\label{estono1}
\end{align}
while if $\tau=0$ and $g \in L^2(0,T;V)$, then we have that
\begin{align}
	& \left| \int_\Omega \bigl\{ g(s)+u_\varepsilon(s) -\pi\bigl( u_\varepsilon(s) \bigr) \bigr\} 
	\partial_t u_\varepsilon(s) dx  \right| \nonumber \\
	&  \quad \le \frac{1}{2} \bigl| \partial_t u_\varepsilon(s) \bigr|_{V^*}^2 
	+ C_2 \bigl( 1+ \bigl| g(s) \bigr|_{V}^2 + \bigl|  u_\varepsilon(s) \bigr|_{V}^2 \bigr)
	\label{estono2}
\end{align}
for a.a.\ $s \in (0,T)$, 
where $C_2$ is a positive constant, depending on $L$.
On the other hand, if $\tau=0$ and $g \in H^1(0,T;H)$, \pier{then we treat separately the related term 
and integrate by parts in time, so to obtain}
\begin{equation}
	\int_0^t\!\!\int _\Omega g(s) \partial_t u_\varepsilon(s)  dxds 
	= \int_\Omega g(t)u_\varepsilon(t) dx 
	-\int_\Omega g(0)u_0 dx 
	- \int_0^t\!\!\int_\Omega \partial_t g (s)  u_\varepsilon(s) dxds
	\label{ibp}
\end{equation} 
\pier{on the right-hand side}, when integrating (\ref{ineumu}) from $0$ to $t \in[0,T]$. 
\pier{Then, one has to estimate the three terms above.}

Hence, \pier{let us integrate (\ref{ineumu}) with respect to time.  In the case $\tau >0$, we use 
(\ref{estong}) and (\ref{estono1})}
to find that there exists a positive constant $C_3$, depending on $C_1$, such that
\begin{align}
	& \frac{1}{2} | \partial_t u_\varepsilon |_{L^2(0,t;V^*)}^2 
	+
	\frac{\tau}{2} | \partial_t u_\varepsilon |_{L^2(0,T;H)}^2 
	+ 
	\frac{1}{2} | \partial_t u_{\Gamma, \varepsilon} |_{L^2(0,t;H_\Gamma)}^2 
	+ 
	\frac{1}{2} \bigl| u_\varepsilon (t) \bigr|_V^2 
	\nonumber \\
	& \quad  {} 
	+ 
	\pier{\frac{\kappa}{2} \int_\Gamma \bigl| \nabla _\Gamma  \pier{u_{\Gamma, \varepsilon}} (t) \bigr|^2 d\Gamma+ 
	\int_\Omega \widehat{\beta}_\varepsilon \bigl( u_\varepsilon(t) \bigr) dx 
	+
	\int_\Gamma \widehat{\beta}_{\Gamma, \varepsilon}\bigl( u_{\Gamma, \varepsilon} (t)\bigr) d\Gamma}
	\nonumber \\
	&\le 
	\pier{\frac{1}{2} | u_0 |_V^2 
	+ 
	\frac{\kappa}{2} \int_\Gamma | \nabla _\Gamma u_{0\Gamma} |^2 d\Gamma
	+ 
	\int_\Omega \widehat{\beta} ( u_0) dx 
	+
	\int_\Gamma \widehat{\beta}_{\Gamma}( u_{0\Gamma} ) d\Gamma
	+ |g_\Gamma |_{L^2(0,T;H_\Gamma)} ^2}
	\nonumber \\
& \quad 
	\pier{{}+	
	C_3 \int_0^t \bigl( 1+ \bigl|  u_{\varepsilon}(s) \bigr|_{V}^2 \bigr)ds
	+
       \frac{3}{2\tau} | g |_{L^2(0,T;H)}^2 
	+ \frac{C_2}{\tau} \int_0^t \bigl( 1+ \bigl|  u_\varepsilon(s) \bigr|_{H}^2 \bigr) ds 
	,}
	\label{last}
\end{align}
where we used the linear continuity of the trace operator from $V$ to $H_\Gamma$
and 
the fundamental property of the {M}oreau--{Y}osida regularizations
\begin{equation*}
	0 \le \widehat{\beta }_\varepsilon (r) \le \widehat{\beta }(r), \quad 
	0 \le \widehat{\beta }_{\Gamma,\varepsilon} (r) \le \widehat{\beta }_{\Gamma }(r)
	\quad \pier{\mbox{for all } \varepsilon \in (0,1] \hbox{ and } r \in \mathbb{R}}.
\end{equation*}
Thus, we can apply the {G}ronwall lemma and \pier{infer} the estimate (\ref{lemma1})\pier{.}
In the case when $\tau=0$, \pier{if $g \in L^2(0,T;V)$,} then \pier{we use \eqref{estono2} and the last two terms of the estimate (\ref{last}) modify into
\begin{equation*}
    C_2 |g|_{L^2(0,T;V)}^2 + C_2 \int_0^t \bigl( 1+ \bigl|  u_\varepsilon(s) \bigr|_{V}^2 \bigr) ds .
\end{equation*}
Thus,} the estimate (\ref{lemma1b}) can be obtained \pier{still} by applying the {G}ronwall lemma, observing that 
\pier{$C_2$ is independent of $\tau$.  Instead, if $g \in H^1(0,T;H)$,
 then in the right-hand side in the estimate (\ref{last}) we can find the terms} 
\begin{equation*}
	\frac{1}{4}\bigl| u_\varepsilon(t) \bigr|_V^2 
	+
	|g|_{C([0,T];H)}^2 
	+ 
	|g|_{C([0,T];H)}|u_0|_H
	+ 
	\frac{1}{2}
	|\partial _t g|_{L^2(0,T;H)}^2 
	+\pier{\frac{1}{2}\int_0^t \big| u_\varepsilon(s) \bigr|_V^2ds},
\end{equation*}
\pier{derived from \eqref{ibp} by applying the Young inequality and the embedding  
inequality $\takeshi{|} z \takeshi{|}_H \leq \takeshi{|} z \takeshi{|}_V$, $z\in V$.
Then,  the estimate (\ref{lemma1b}) follows easily also in this case.}
\end{proof}

\pier{In the next lemmas we will not deal with constants as upperbounds, but with time functions whose 
boundedness in $L^2(0, T)$ (or  $L^\infty (0, T)$) is understood to be uniform with respect to 
$\varepsilon$ and $\kappa \in (0,1]$.}

\begin{lemma} 
There exists a function $\Lambda_0$, bounded in $L^2(0,T)$, such that
\begin{equation}
	\bigl| \mu_\varepsilon(t)-m \bigl( \mu_\varepsilon(t) \bigr) \bigr|_{V} 
	\le \Lambda_0(t)
	\label{32}
\end{equation}
for a.a.\ $t \in (0,T)$. 
\end{lemma}

\begin{proof}
\pier{We recall  the Poincar\'e inequality \eqref{pier3} and test (\ref{chd1e}) by $\mu_\varepsilon(t)-m(\mu_\varepsilon(t)) \in V_0$. 
In view of (\ref{lemma1}) or (\ref{lemma1b}), we can deduce that 
\begin{align*}
	\bigl| \mu_\varepsilon(t)-m \bigl( \mu_\varepsilon(t) \bigr) \bigr|_{V}^2 
	& \le \bigl( C_{\rm P} \bigl| \nabla \bigl( \mu_\varepsilon(t)-m \bigl( \mu_\varepsilon(t) \bigr) \bigr) \bigr|_H \bigr)^2 \\
	&=  C_{\rm P}^{\,2}  \takeshi{ \bigl \langle}  \partial_t u_\varepsilon(t), \mu_\varepsilon(t)-m \bigl( \mu_\varepsilon(t) \bigr)  \takeshi{ \bigr \rangle}_{V^*,V}  \\
	& \le C_{\rm P}^{\, 2 }\bigl|    \partial_t u_\varepsilon (t)  \bigr|_{V^*} 
	\bigl| \mu_\varepsilon(t)-m \bigl( \mu_\varepsilon(t) \bigr) \bigr|_{V}
\end{align*}
for a.a.\ $t \in (0,T)$, whence \eqref{32} follows with 
$\Lambda_0 =   C_{\rm P}^{\,2} \takeshi{|} \partial_t u_\varepsilon \takeshi{|}_{V^*} $.}
\end{proof}

\begin{lemma} 
There exists a function $\Lambda_1$, bounded in $L^2(0,T)$, such that
\begin{equation}
	\bigl| \beta_\varepsilon \bigl( u_\varepsilon (t) \bigr) \bigr|_{L^1(\Omega)}
	+  
	\bigl| \beta_{\Gamma,\varepsilon} \bigl( u_{\Gamma,\varepsilon} (t) \bigr) \bigr|_{L^1(\Gamma)}
	\le \Lambda_1(t)
	\label{33}
\end{equation}
for a.a.\ $t \in (0,T)$. 
\end{lemma}

\begin{proof}
\pier{First}, we recall the useful inequality \pier{proved in} \cite[Section~5]{GMS09} 
\pier{and holding for both graphs} under the assumptions (A2) and (A4): 
there exist two positive constants $\delta_0$ and $c_1$ such that
\begin{equation}
	\beta_\varepsilon(r)(r-m_0)\ge \delta_0 \bigl| \beta_\varepsilon(r) \bigr|-c_1, \quad 
	\beta_{\Gamma,\varepsilon}(r)(r-m_0)\ge \delta_0 \bigl| \beta_{\Gamma,\varepsilon}(r) \bigr|-c_1
	\label{gms}
\end{equation} 
for all $r \in \mathbb{R}$ and $\varepsilon \in (0,1]$. 
\pier{Then, we test} (\ref{chd2e}) by $u_\varepsilon-m_0$ and take advantage of  (\ref{chd5e}) \pier{in order to} deduce that 
\begin{align*}
	& \delta_0 \int_\Omega \bigl| \beta_\varepsilon ( u_\varepsilon )\bigr| dx 
	+ 
	\delta_0 \int_\Gamma \bigl| \beta_{\Gamma, \varepsilon} ( u_{\Gamma,\varepsilon} )\bigr|d\Gamma 
	- c_1 \bigl( |\Omega|+|\Gamma|\bigr) 
	\nonumber \\
	& \le \int_\Omega \beta_\varepsilon(u_\varepsilon) ( u_\varepsilon -m_0 ) dx 
	+  \int_\Gamma \beta_{\Gamma,\varepsilon} (u_{\Gamma,\varepsilon}) ( u_{\Gamma, \varepsilon}-m_0 ) d\Gamma
	\nonumber \\
	& \le \int_\Omega \bigl( \mu_\varepsilon -m ( \mu_\varepsilon)\bigr)
	( u_\varepsilon - m_0 ) dx 
	+ \int_\Omega \bigl( g-\tau \partial_t u_\varepsilon-\pi(u_\varepsilon) \bigr) (u_\varepsilon-m_0)dx 
	\nonumber \\
	& \quad {}
	- \int_\Omega |\nabla u_\varepsilon|^2 dx 
	+ \int_\Gamma \bigl( g_\Gamma - \partial_t u_{\Gamma,\varepsilon}-\pi_\Gamma 
	(u_{\Gamma,\varepsilon}) \bigr) (u_{\Gamma,\varepsilon}-m_0)d\Gamma 
	-\kappa \int_\Gamma |\nabla _\Gamma u_{\Gamma,\varepsilon}|^2 d\Gamma
\end{align*}
a.e.\ on (0,T), 
where based on (\ref{masse}) we used the following equality:
\begin{equation*}
	\int_\Omega m \bigl( \mu_\varepsilon(t) \bigr)\bigl( u_\varepsilon(t) - m_0 \bigr) dx 
	= m \bigl( \mu_\varepsilon(t) \bigr)\int_\Omega\bigl( u_\varepsilon(t) - m_0 \bigr) dx=0
\end{equation*}
for a.a.\ $t \in (0,T)$. 
Hence, by squaring we \pier{arrive at}
\begin{align*}
	& \pier{\delta_0^{\,2}} \left( \int_\Omega \bigl| \beta_\varepsilon ( u_\varepsilon )\bigr| dx 
	+ 
	 \int_\Gamma \bigl| \beta_{\Gamma, \varepsilon} ( u_{\Gamma,\varepsilon} )\bigr|d\Gamma 
	 \right)^{\!\!2}
	\nonumber \\
	& \le 
	4c_1^{\,2} \bigl( |\Omega|+|\Gamma|\bigr)^{\!2 }
	+
	4\bigl| \mu_\varepsilon -m ( \mu_\varepsilon)\bigr|_{\pier V}^2
	|u_\varepsilon - m_0|_{\pier H} ^2
	\nonumber \\
	& \quad {}
	+ 4\left(
	|g|_H+ \tau |\partial_t u_\varepsilon|_H+\sqrt{2} \bigl( L | u_\varepsilon |_H
	+ \bigl| \pi(0) \bigr| |\Omega|^{1/2} \bigr) \right) ^{\!2}
	|u_\varepsilon-m_0|_{\pier H}^2
	\nonumber \\ 
	& \quad {}+ 4 \left(  |g_\Gamma|_{H_\Gamma} 
	+ |\partial_t u_{\Gamma,\varepsilon}|_{H_\Gamma}
	+\sqrt{2} \bigl( L_\Gamma | u_{\Gamma,\varepsilon}|_{H_\Gamma}
	+ \bigl| \pier{\pi_\Gamma }(0) \bigr| |\Gamma|^{1/2} \bigr) \right) ^{\!2} 
	|u_{\Gamma,\varepsilon}-m_0|_{H_\Gamma} ^2
\end{align*}
a.e.\ on (0,T). \pier{Now, note that \eqref{lemma1} or \eqref{lemma1b}, \eqref{32} and  assumption~(A5)
enable us to infer that the right-hand side of the last inequality is a summable function in $(0,T)$.
Hence, it follows that there is a function 
$\Lambda_1$, bounded in $L^2(0,T)$,} such that (\ref{33}) holds. 
\end{proof}

\begin{lemma} 
There exist functions $\Lambda_2$ and $\Lambda_3$, bounded in $L^2(0,T)$, such that
\begin{gather}
	\bigl| m \bigl(\mu_\varepsilon (t) \bigr) \bigr| \le \Lambda_2(t), 
	\label{34-1} \\
	\bigl| \mu_\varepsilon (t)  \bigr|_{V} 
	\le \Lambda_3(t)
	\label{34-2}
\end{gather}
for a.a.\ $t \in (0,T)$. 
\end{lemma}

\begin{proof}
Integrating (\ref{chd2e}) over $\Omega$ directly \pier{and using (\ref{chd5e}) and (\ref{masse}) lead to} 
\begin{align*}
	|\Omega| m(\mu_\varepsilon) 
	& = \int_\Omega \pier{\bigl(
	 \beta_\varepsilon (u_\varepsilon) 
	+ \pi (u_\varepsilon) 
	- g \bigr)} dx
	\\ 
	& \quad 
	{} + \int_\Gamma 
	\bigl( \partial_t u_{\Gamma, \varepsilon} 
	+ \beta_{\Gamma, \varepsilon} (u_{\Gamma, \varepsilon} )
	+ \pi_\Gamma (u_{\Gamma, \varepsilon}) 
	- g_\Gamma \bigr) d\Gamma 
	\nonumber \\
	& \le 
	\bigl| \beta_\varepsilon (u_\varepsilon) \bigr|_{L^1(\Omega)}
	+ L |\pier{u_\varepsilon}|_H |\Omega|^{1/2} 
	+ \bigl| \pi (0) \bigr| |\Omega|
	+ |g|_{H} |\Omega|^{1/2}
	+ |\partial_t u_{\Gamma, \varepsilon}|_{H_\Gamma} |\Gamma|^{1/2} 
	\nonumber \\
	& \quad {}
	+ \bigl| \beta_{\Gamma, \varepsilon} (u_{\Gamma, \varepsilon} ) \bigr|_{L^1(\Gamma)}
	+ L_\pi |\pier{u_{\Gamma, \varepsilon}}|_{H_\Gamma} |\Gamma|^{1/2} 
	+ \bigl| \pi_\Gamma (0) \bigr| |\Gamma| 
	+ |g_\Gamma |_{H_\Gamma}|\Gamma|^{1/2}
\end{align*}
a.e.\ on $(0,T)$, 
whence Lemmas~3.1 and 3.3 allow us to \pier{conclude that a function $\Lambda_2$, bounded in $L^2(0,T)$, 
exists such that (\ref{34-1}) holds.
Next, we can combine (\ref{32}) and (\ref{34-1}) to deduce that 
\begin{align*}
	\bigl| \mu_\varepsilon(t) \bigr|_V 
	& \le \bigl| \mu_\varepsilon(t) - m\bigl(\mu_\varepsilon(t) \bigr)\bigr|_V
	+ \bigl| m \bigl(\mu_\varepsilon (t) \bigr) \bigr|_V
	\nonumber \\
	&  \le \Lambda_0(t) +  |\Omega|^{1/2}  \Lambda_2(t) =:\Lambda_3(t)
\end{align*}
for a.a.\ $t \in (0,T)$, where the function $\Lambda_3$ is bounded in $L^2(0,T)$ as well.} 
\end{proof}

\begin{lemma} 
There \pier{exists} a function $\Lambda_4$, bounded in $L^2(0,T)$, such that
\begin{equation}
	\bigl| \beta_\varepsilon \bigl( u_\varepsilon (t) \bigr) \bigr|_{H}
	+  
	\bigl| \beta_{\varepsilon} \bigl( u_{\Gamma,\varepsilon} (t) \bigr) \bigr|_{H_\Gamma}
	\le \Lambda_4(t)
	\label{35}
\end{equation}
for a.a.\ $t \in (0,T)$. 
\end{lemma}

\begin{proof}
We \pier{test (\ref{chd2e}) by $\beta_\varepsilon(u_\varepsilon)$ and exploit (\ref{chd5e}) to obtain}
\begin{align}
	& \int_\Omega \bigl| \beta_\varepsilon (u_\varepsilon)\bigr|^2dx + \int _\Omega \beta'_\varepsilon(u_\varepsilon) |\nabla u_\varepsilon|^2dx 
	\nonumber \\
	& \quad {} 
	+\int_\Gamma \beta_{\Gamma, \varepsilon}(u_{\Gamma, \varepsilon}) 
	\beta_\varepsilon(u_{\Gamma, \varepsilon}) d\Gamma
	+ \kappa \int_\Gamma \beta'_{\Gamma, \varepsilon}(u_{\Gamma, \varepsilon}) 
	|\nabla_\Gamma u_{\Gamma,\varepsilon} |^2 d\takeshi{\Gamma} 
	\nonumber \\
	& \le \int_\Omega \bigl( \mu_\varepsilon+g-\tau \partial _t u_\varepsilon -\pi(u_\varepsilon)\bigr) \beta_\varepsilon(u_\varepsilon) dx
	+\int_\Gamma \bigl( g_\Gamma-\partial _t u_{\Gamma,\varepsilon} -\pi_\Gamma(u_{\Gamma,\varepsilon})\bigr) 
	\beta_\varepsilon(u_{\Gamma,\varepsilon}) d\Gamma
	\label{35est}
\end{align}
a.e.\ in $(0,T)$, where we used that the trace of $\beta_\varepsilon(u_\varepsilon)$ is 
equal to $\beta_\varepsilon(u_{\Gamma,\varepsilon})$. 
In order to treat the gap between $\beta$ and $\beta_\Gamma$, \pier{we recall  (\ref{ccconde}) and observe that}
\begin{align*}
	\int_\Gamma \beta_{\Gamma,\varepsilon}(u_{\Gamma,\varepsilon}) 
	\beta_\varepsilon (u_{\Gamma,\varepsilon}) d\Gamma 
	& = 	\int_\Gamma 
	\bigl| \beta_{\Gamma,\varepsilon}(u_{\Gamma,\varepsilon}) \bigr|
	\bigl| \beta_\varepsilon(u_{\Gamma,\varepsilon}) \bigr| d\Gamma 
	\nonumber \\
	& \ge \frac{1}{\varrho}\int_\Gamma 
	\bigl| \beta_\varepsilon (u_{\Gamma,\varepsilon}) \bigr|^2 d\Gamma 
	- \frac{c_0}{\varrho} \int_\Gamma 
	\bigl| \beta_\varepsilon (u_{\Gamma,\varepsilon}) \bigr|d\Gamma 
	\nonumber \\
	& \ge \frac{1}{2\varrho} \int_\Gamma \bigl| \beta_\varepsilon(u_{\Gamma,\varepsilon})\bigr|^2 
	\takeshi{d \Gamma}	
	- \frac{c_0^{\,2}}{2\varrho}|\Gamma|
\end{align*}
a.e.\ in $(0,T)$, because $\beta_\varepsilon$ and $\beta_{\Gamma,\varepsilon}$ \pier{have} the same sign. 
Therefore, applying the {Y}oung inequality \pier{in} (\ref{35est}) we deduce that 
\begin{align*}
	& \bigl| \beta_\varepsilon (u_\varepsilon)\bigr|_{H} ^2
	+  \frac{1}{2\varrho}\bigl| \beta_\varepsilon(u_{\Gamma,\varepsilon})\bigr|_{H_\Gamma}^2
	\nonumber \\
	& \le \frac{5}{2} \bigl( | \mu_\varepsilon|_H^2 
	+|g|_H^2+\tau^2 |\partial _t u_\varepsilon|_H^2 + 
	L^2 | u_\varepsilon |_H^2 + \bigl| \pi(0) \bigr|^2 |\Omega| \bigr)
	+
	\frac{1}{2} \bigl| \beta_\varepsilon (u_\varepsilon)\bigr|_{H} ^2
	\nonumber \\
	&\quad  {}
	+ 4\varrho \bigl( |g_\Gamma|_{H_\Gamma}^{\takeshi{2}}
	+ 
	|\partial _t u_{\Gamma,\varepsilon}|_{H_\Gamma} ^{\pier 2}
	+ 
	L^2_\Gamma | u_{\Gamma,\varepsilon}|_{H_\Gamma}^2
	+\bigl| \pi_\Gamma(0) \bigr|^2 |\Gamma| \bigr) 
	+ \frac{1}{4\varrho} \bigl| \beta_\varepsilon(u_{\Gamma,\varepsilon}) \bigr| _{H_\Gamma}^2
\end{align*}
a.e.\ in $(0,T)$, that is, 
by virtue of Lemmas~3.1 and 3.4, there \pier{is} a function $\Lambda_4$, which is bounded in $L^2(0,T)$, 
such that (\ref{35}) holds.
\end{proof}

\begin{lemma} 
There exist functions $\Lambda_5$ and $\Lambda_6$, bounded in $L^2(0,T)$, such that
\begin{gather}
	\bigl| \Delta  u_\varepsilon (t) \bigr|_{H} \le \Lambda_5(t), 
	\label{36-1} \\
	\sqrt{\kappa} \bigl| \partial_{\boldsymbol{\nu}} u_\varepsilon (t)  \bigr|_{H_\Gamma} 
	+ \bigl| \partial_{\boldsymbol{\nu}} u_\varepsilon (t)  \bigr|_{W_\Gamma^*}  
	\le \Lambda_6(t)
	\label{36-2}
\end{gather}
for a.a.\ $t \in (0,T)$. 
\end{lemma}

\begin{proof}
We write the equation (\ref{chd2e}) as 
\begin{equation}
	-\Delta u_\varepsilon = \mu_\varepsilon +g - \tau \partial_t u_\varepsilon- \beta_\varepsilon(u_\varepsilon) 
	- \pi (u_\varepsilon)
	\quad \text{a.e.\ in }Q, 
	\label{2re}
\end{equation}
and observe that, by a comparison in (\ref{2re}) and recalling (\ref{34-2}), (A5), (\ref{lemma1}) (or (\ref{lemma1b}) if $\tau=0$), and (\ref{35}), 
we deduce that there \pier{is} a function $\Lambda_5$, which is bounded in $L^2(0,T)$, such that 
(\ref{36-1}) holds. Next, 
we recall \cite[Theorem~3.2, p.~1.79]{BG87} in order to \pier{claim} that there exists a positive constant 
$C_5$ such that
\begin{equation*}
	\bigl| u_\varepsilon(t) \bigr|_{H^{3/2}(\Omega)} 
	\le C_5 \bigl\{ \bigl| u_{\Gamma,\varepsilon} (t) \bigr|_{V_\Gamma} 
	+ \bigl| \Delta u_\varepsilon(t) \bigr|_H \bigr\}
\end{equation*}
for a.a.\ $t \in (0,T)$, whence we can exploit (\ref{lemma1}) (or (\ref{lemma1b}) if $\tau=0$) and (\ref{36-1}) \pier{along with 
\cite[Theorem~2.27, p.~1.64]{BG87} to derive the estimate  (\ref{36-2}). Concerning (\ref{36-2}), we remark that the coefficient $\sqrt{\kappa}$ is only in the first term since the control of  $\takeshi{|} \partial_{\boldsymbol{\nu}} u_\varepsilon (t) \takeshi{|}_{W_\Gamma^*}$ just needs the bound of $ \takeshi{|} u_\varepsilon(t) \takeshi{|}_{V}$ and \eqref{36-1}.}
\end{proof}

\begin{lemma} 
There exist functions $\Lambda_7$, $\Lambda_8$, $\Lambda_{10}$,  \pier{$\Lambda_{11}$,  which are bounded in $L^2(0,T)$, 
and  $\Lambda_9$,} which is bounded in $L^\infty(0,T)$, such that 
\begin{gather}
	\sqrt{\kappa} \bigl| \beta_{\Gamma,\varepsilon} \bigl( u_{\Gamma,\varepsilon} (t) \bigr) \bigr|_{H_\Gamma} \le \Lambda_7(t), 
	\label{37-1} \\
	\kappa^{3/2} \bigl| \Delta_\Gamma u_{\Gamma,\varepsilon} (t)  \bigr|_{H_\Gamma}  
	\le \Lambda_8(t),
	\label{37-2} \\
	\sqrt{\kappa} \bigl| \Delta_\Gamma u_{\Gamma,\varepsilon} (t)  \bigr|_{V_\Gamma^*} 
	\le \Lambda_9(t),
	\label{37-3} \\
	\pier{\kappa} \bigl| \Delta_\Gamma u_{\Gamma,\varepsilon} (t)  \bigr|_{W_\Gamma^*}  
	\le \Lambda_{10}(t),
	\label{37-4} \\
	\bigl| \beta_{\Gamma,\varepsilon} \bigl( u_{\Gamma,\varepsilon} (t) \bigr) \bigr|_{W_\Gamma^*}
	\le \Lambda_{11}(t)
	\label{37-5} 
\end{gather}
for a.a.\ $t \in (0,T)$. 
\end{lemma}

\begin{proof}
We write (\ref{chd5e}) as 
\begin{equation*}
	-\kappa \Delta_\Gamma u_{\Gamma,\varepsilon}+\beta_{\Gamma,\varepsilon}(u_{\Gamma,\varepsilon}) 
	= g_\Gamma -\partial_t u_{\Gamma,\varepsilon}-\partial _{\boldsymbol{\nu}}u_\varepsilon-\pi_\Gamma (u_{\Gamma,\varepsilon})
	\quad \text{a.e.\ on } \Sigma.
\end{equation*}
Multiplying it by $\beta_{\Gamma, \varepsilon}(u_{\Gamma,\varepsilon})$ \pier{and integrating over $\Gamma$, we}  obtain 
\begin{align}
	& \kappa \int_\Gamma \beta_{\Gamma,\varepsilon}'(u_{\Gamma,\varepsilon}) 
	|\nabla_\Gamma u_{\Gamma,\varepsilon} |^2 d\Gamma
	+ \bigl| \beta_{\Gamma,\varepsilon} (u_{\Gamma,\varepsilon}) \bigr|_{H_\Gamma}^2 
	\nonumber \\
	& \le 
	\bigl| g_\Gamma -\partial_t u_{\Gamma,\varepsilon}-\pi_\Gamma (u_{\Gamma,\varepsilon})\bigr|_{H_\Gamma}
	\bigl| \beta_{\Gamma,\varepsilon} (u_{\Gamma,\varepsilon}) \bigr|_{H_\Gamma}
	+
	| \partial _{\boldsymbol{\nu}} u_\varepsilon |_{H_\Gamma}
	\bigl| \beta_{\Gamma,\varepsilon} (u_{\Gamma,\varepsilon}) \bigr|_{H_\Gamma}
\end{align}
a.e.\ in $(0,T)$. 
Next, using the {Y}oung inequality in both terms of the right hand side, \pier{we} find that 
\begin{equation*}
	\frac{1}{2}
	\bigl| \beta_{\Gamma,\varepsilon} (u_{\Gamma,\varepsilon}) \bigr|_{H_\Gamma}^2
	\le \bigl| g_\Gamma -\partial_t u_{\Gamma,\varepsilon}-\pi_\Gamma (u_{\Gamma,\varepsilon})
	\bigr|_{H_\Gamma}^2
	+
	| \partial _{\boldsymbol{\nu}} u_\varepsilon |_{H_\Gamma}^2
\end{equation*}
whence \pier{(A5),} (\ref{lemma1}) (or (\ref{lemma1b}) if $\tau=0$)\pier{, (A3)} and (\ref{36-2}) enable us to deduce that 
\begin{align*}
	\sqrt{\kappa}
	\bigl| \beta_{\Gamma,\varepsilon} \bigl( u_{\Gamma,\varepsilon} (t)\bigr) \bigr|_{H_\Gamma}
	& \le \sqrt{2} \left( \bigl| g_\Gamma(t) \bigr|_{H_\Gamma}
	+  \bigl| \partial_t u_{\Gamma,\varepsilon}(t) \bigr|_{H_\Gamma}
	+  \bigl| \pi_\Gamma \bigl(u_{\Gamma,\varepsilon}(t) \bigr)
	\bigr|_{H_\Gamma}
	+\sqrt{\kappa}
	\bigl| \partial _{\boldsymbol{\nu}} u_\varepsilon(t) \bigr|_{H_\Gamma} \right) 
\end{align*}
for a.a.\ $t \in (0,T)$, \pier{as $0 < \kappa \leq 1$ in our setting. The above right-hand side is 
uniformly bounded in $L^2(0,T)$, then there is a function $\Lambda_7$ such that \eqref{37-1} holds.}

\pier{Next,} multiplying (\ref{chd5e}) by $\sqrt{\kappa}$ and comparing the terms, we obtain 
\begin{align*}
	\kappa^{3/2}  \bigl| \Delta_\Gamma u_{\Gamma,\varepsilon} (t)  \bigr|_{H_\Gamma}  
	& \le \sqrt{\kappa} \bigl|  \beta_{\Gamma,\varepsilon} \bigl( u_{\Gamma,\varepsilon} (t)\bigr)
	\bigr|_{H_\Gamma}  
	+ \bigl| 
	 g_\Gamma(t) 
	 \bigr|_{H_\Gamma} 
	 + \bigl| \partial_t 
	u_{\Gamma,\varepsilon} (t)  \bigr|_{H_\Gamma}  
	\nonumber \\
	& \quad  {} 
	+ \sqrt{\kappa} \bigl| 
	\partial _{\boldsymbol{\nu}} u_{\varepsilon} (t)  \bigr|_{H_\Gamma}
	+
	\bigl|  \pi_{\Gamma} \bigl( u_{\Gamma,\varepsilon} (t)\bigr)
	\bigr|_{H_\Gamma}  
	\nonumber \\
	& \le \Lambda_8(t)
\end{align*}
for a.a.\ $t \in (0,T)$, 
where $\Lambda_8$ is bounded in $L^2(0,T)$. 
Moreover, since $\Delta_\Gamma$ is a linear and bounded operator from $V_\Gamma$ to 
$V_\Gamma^*$, then from  (\ref{lemma1}) (or (\ref{lemma1b}) if $\tau=0$) it follows that 
\begin{equation*}
	\sqrt{\kappa} \bigl| \Delta_\Gamma u_{\Gamma, \varepsilon} (t)\bigr|_{V_\Gamma^*} 
	\le \sqrt{\kappa} C_6 \, \bigl|  u_{\Gamma, \varepsilon} (t)\bigr|_{V_\Gamma} 
	\le \Lambda_9(t)
\end{equation*}
for a.a.\ $t \in (0,T)$, 
where $C_6$ is a positive constant and $\Lambda_9$ is bounded in $L^\infty (0,T)$. 
Then, by interpolation 
$H_\Gamma \mathop{\hookrightarrow} W_\Gamma^* \mathop{\hookrightarrow} V_\Gamma^*$, 
there exists a positive constant $C_7$ such that 
\begin{equation*}
	| \kappa \Delta_\Gamma u_{\Gamma,\varepsilon} |_{W_\Gamma^*}
	\le C_7 \bigl| \kappa^{3/2} \Delta_\Gamma u_{\Gamma,\varepsilon} \bigr|_{H_\Gamma}^{1/2}
	\bigr| \sqrt{\kappa} \Delta_\Gamma u_{\Gamma,\varepsilon} \bigr|_{V_\Gamma^*}^{1/2}
\end{equation*}
a.e.\ on $(0,T)$. 
Therefore, we can find \pier{a function $\Lambda_{10}$, depending on 
$\Lambda_8$ and $\Lambda_9$,  which is bounded in $L^2 (0,T)$ (actually, up to $L^4(0,T)$) such that (\ref{37-4})} holds. 
Consequently, we recall \pier{(\ref{37-4})}, (A5),  (\ref{lemma1}) (or (\ref{lemma1b}) if $\tau=0$), (\ref{36-2})
and make a comparison of terms in (\ref{chd5e}) 
to deduce that 
\begin{align*}
	\bigl| \beta_{\Gamma,\varepsilon} \bigl( u_{\Gamma,\varepsilon} (t)\bigr) \bigr|_{W_\Gamma^*}
	& \le 
	\bigl| \kappa \Delta_\Gamma u_{\Gamma,\varepsilon}(t) \bigr|_{W_\Gamma^*}
	+
	\bigl| g_\Gamma(t) \bigr|_{W_\Gamma^*}
	+  \bigl| \partial_t u_{\Gamma,\varepsilon}(t) \bigr|_{W_\Gamma^*}
	\nonumber \\
	& \quad {}
	+  \bigl| \pi_\Gamma \bigl(u_{\Gamma,\varepsilon}(t) \bigr)
	\bigr|_{W_\Gamma^*}
	+
	\bigl| \partial _{\boldsymbol{\nu}} u_\varepsilon(t) \bigr|_{W_\Gamma^*}
	\nonumber \\
	& \le \Lambda_{11} (t)
\end{align*}
for a.a.\ $t \in (0,T)$, 
\pier{with  $\Lambda_{11}$ bounded in $L^2(0,T)$  so that (\ref{37-5}) holds.}
\end{proof}

\subsection{Proof of Proposition~2.1}
\pier{For the proof of the complete statement we refer to \cite{CGS14, CF15b}. 
Here we just detail the limit procedure
as $\varepsilon \searrow 0$, so as to show existence of the solution. 
Of course, we keep $\kappa \in (0,1]$ fixed and let the triplet $(u_\varepsilon, \mu_\varepsilon, u_{\Gamma,\varepsilon})$
solve \eqref{chd1e}--\eqref{chd6e} for $\varepsilon >0$.} By virtue of Lemmas~3.1 \pier{and~3.4--3.7, there exist some limit functions 
$u_\kappa, \mu_\kappa, u_{\Gamma, \kappa}, \xi_\kappa, \xi_{\Gamma, \kappa}$ such that
\begin{align*}
	u_\varepsilon \to u_\kappa 
	& \quad \text{weakly star in } H^1(0,T;V^*) \cap L^\infty(0,T;V), \\
	\sqrt{\tau} u_\varepsilon \to \sqrt{\tau} u_\kappa 	
	& \quad \text{weakly  in } H^1(0,T;H), \\
		\mu_\varepsilon \to \mu_\kappa 
	&  \quad \text{weakly in } L^2(0,T;V), \\
	u_{\Gamma,\varepsilon} \to u_{\Gamma, \kappa}
	&  \quad \text{weakly star in } H^1(0,T;H_\Gamma) \cap L^\infty ( 0,T;V_\Gamma ), \\
	\Delta u_\varepsilon \to \Delta u_\kappa
	& \quad \text{weakly in } L^2(0,T;H), \\
	\Delta_\Gamma u_{\Gamma,\varepsilon} \to \Delta_\Gamma u_{\Gamma, \kappa}
	& \quad \text{weakly in } L^2(0,T;H_\Gamma), \\
	\partial_{\boldsymbol{\nu}} u_\varepsilon \to \partial_{\boldsymbol{\nu}} u_\kappa 
	& \quad \text{weakly in } L^2(0,T;H_\Gamma), \\
	\beta_\varepsilon( u_\varepsilon) \to \xi_\kappa 
	& \quad \text{weakly in } L^2(0,T;H), \\
	\beta_{\Gamma,\varepsilon}(u_{\Gamma,\varepsilon}) \to \xi_{\Gamma, \kappa} 
	& \quad \text{weakly in } L^2(0,T;H_\Gamma)
\end{align*}
as $\varepsilon \searrow 0$, in principle for a subsequence, then for the entire family due to the uniqueness of the limits}.
Moreover, from a \pier{well-known} compactness theorem (see, e.g.\ \cite[Section~8, Corollary~4]{Sim87}), 
we obtain
\begin{align*}
	u_\varepsilon \to u_\kappa 
	& \quad \text{strongly in } C\bigl( [0,T];H \bigr) \cap L^2(0,T;V), \\
	u_{\Gamma,\varepsilon} \to u_{\Gamma, \kappa} 
	& \quad \text{strongly in } C\bigl( [0,T];H_\Gamma \bigr) \cap L^2(0,T;V_\Gamma)
\end{align*}
as $\varepsilon \searrow 0$. This and the {L}ipschitz continuities of $\pi$ and $\pi_\Gamma$ give us 
\begin{align*}
	\pi(u_\varepsilon) \to \pi(u_\kappa) 
	& \quad \text{strongly in } C\bigl( [0,T];H \bigr), \\
	\pi_\Gamma(u_{\Gamma,\varepsilon}) \to \pi_\Gamma(u_{\Gamma, \kappa}) 
	& \quad \text{strongly in } C\bigl( [0,T];H_\Gamma \bigr)
\end{align*}
as $\varepsilon \searrow 0$. 
Now, applying \cite[Proposition~2.2, p.~38]{Bar10} \pier{it is straightforward} deduce that 
\begin{equation*}
	\xi_\kappa \in \beta(u_\kappa) \quad \text{a.e.\ in } Q, \quad 
	\xi_\Gamma \in \beta_\Gamma(u_{\Gamma, \kappa}) \quad \text{a.e.\ on } \Sigma.
\end{equation*}
Therefore, passing to the limit as $\varepsilon \searrow 0$ in (\ref{chd1e})--(\ref{chd6e}), 
\pier{we infer that the limit quintuplet  $(u_\kappa, \mu_\kappa, \xi_\kappa,  u_{\Gamma, \kappa}, \xi_{\Gamma, \kappa})$ solves
(\ref{chd1})--(\ref{chd6}) and thus conclude the existence proof.} \hfill $\Box$

\section{Proof of main theorems}
\setcounter{equation}{0}

\pier{In view of Proposition~\ref{prop} and with the aim of summarizing the previous section, we see that 
for  $\kappa \in (0,1]$ the above limit $(u_\kappa, \mu_\kappa, \xi_\kappa,  u_{\Gamma, \kappa}, \xi_{\Gamma, \kappa})$ satisfies}
\begin{align*}
	& u_\kappa \in H^1(0,T;V^*)\cap L^\infty (0,T;V) \cap L^2\bigl( 0,T;H^2(\Omega) \bigr), \\
	& \tau u_\kappa \in H^1(0,T;H) \cap C\bigl( [0,T]; V\bigr), \\
	& \mu_\kappa \in L^2(0,T;V), \quad\xi_\kappa \in L^2(0,T;H), \\
	& u_{\Gamma,\kappa} \in H^1(0,T;H_\Gamma )\cap C \bigl( [0,T];V_\Gamma \bigr) \cap L^2 \bigl( 0,T;H^2(\Gamma) \bigr), \\
	& \xi_{\Gamma,\kappa} \in L^2(0,T;H_\Gamma)
\end{align*}
and solves
\begin{align} 
	& \partial_t u_\kappa - \Delta \mu_\kappa =0 
	\quad \text{a.e.\ in }Q, 
	\label{chd1k}
	\\
	&\mu_\kappa = \tau \partial_t u_\kappa -\Delta u_\kappa 
	+ \xi_\kappa + \pi (u_\kappa)-g, \quad \xi_\kappa \in \beta(u_\kappa)
	\quad \text{a.e.\ in }Q, 
	\label{chd2k}
	\\
	& \partial_{\boldsymbol{\nu}} \mu_\kappa =0 
	\quad  \text{a.e.\ on }\Sigma, 
	\label{chd3k}
	\\
	& (u_{\kappa})_{|_{\Gamma}} =u_{\Gamma,\kappa} 
	\quad  \text{a.e.\ on }\Sigma,
	\label{chd4k}
	\\ 
	&\partial_t u_{\Gamma,\kappa}+ \partial_{\boldsymbol{\nu}} u_\kappa 
	- \kappa \Delta_\Gamma u_{\Gamma,\kappa} 
	+ \xi_{\Gamma,\kappa} 
	+ \pi_\Gamma(u_{\Gamma, \kappa})= g_\Gamma, \quad 
	\xi_{\Gamma, \kappa} \in \beta_{\Gamma, \kappa} (u_{\Gamma,\kappa}) 
	\quad  \text{a.e.\ on }\Sigma, 
	\label{chd5k}
	\\
	& u_\kappa (0)=u_0 \quad \text{a.e.\ in } \Omega, 
	\quad u_{\Gamma,\kappa}(0)=u_{0\Gamma} \quad \text{a.e.\ on }\Gamma.
	\label{chd6k}
\end{align}
\pier{If $\tau=0$, then (\ref{chd1k}) and  (\ref{chd3k}) are replaced by the variational equality}
\begin{equation}
	\bigl \langle \partial_t u_\kappa (t), z \bigr \rangle_{V^*,V} 
	+ \int_\Omega \nabla \mu_\kappa(t)\cdot \nabla z dx =0
	\quad \ \pier{\text{for all } z \in V, \ \, \hbox{for a.a.\ $t \in (0,T)$}}.   
	\label{vf}
\end{equation}
\pier{Moreover, owing to Lemma~3.1 and the weak and weak star lower semicontinuity of norms we have the following uniform estimates: 
if $\tau > 0$, then from \eqref{lemma1} it follows that
%
%
\begin{align}
	& | u_\kappa |_{H^1(0,T;V^*) } +
	| u_\kappa |_{L^\infty(0,T;V)}+
	\sqrt{\tau} | \partial_t u_\kappa |_{L^2(0,T;H)} \nonumber\\
	&+
	| u_{\Gamma, \kappa} |_{H^1(0,T;H_\Gamma)} +
	| u_{\Gamma, \kappa} |_{L^\infty(0,T;W_\Gamma)}+
	\sqrt{\kappa} | u_{\Gamma, \kappa} |_{L^\infty(0,T;V_\Gamma)}
	\le M_1(\tau),
	 \label{e1}
\end{align}
while in the case $\tau =0$, then  by \eqref{lemma1b} we have that
%
%
\begin{align}
	& | u_\kappa |_{H^1(0,T;V^*) } +
	| u_\kappa |_{L^\infty(0,T;V)}+
	| u_{\Gamma, \kappa} |_{H^1(0,T;H_\Gamma)} 	 \nonumber\\
	&+
	| u_{\Gamma, \kappa} |_{L^\infty(0,T;W_\Gamma)}+
	\sqrt{\kappa} | u_{\Gamma, \kappa} |_{L^\infty(0,T;V_\Gamma)}
	\le M_2.
	 \label{pier12}
\end{align}
Furthermore, in both cases} from (\ref{34-2}), (\ref{35}), (\ref{36-1}), (\ref{36-2}), and (\ref{37-5}) we obtain 
\begin{gather}
	 | \mu_\kappa |_{L^2 (0,T;V)} \le \liminf_{\varepsilon \to 0} | \mu_\varepsilon |_{L^2 (0,T;V)} 
	 \le |\Lambda_3|_{L^2(0,T)}, \\
	| \xi_\kappa |_{L^2(0,T;H)} \le \liminf_{\varepsilon \to 0} 
	\bigl| \beta_\varepsilon( u_\varepsilon) \bigr|_{L^2(0,T;H)}
	\le |\Lambda_4|_{L^2(0,T)}, \\
	| \Delta u_{\kappa} |_{L^2(0,T;H)}
	\le \liminf_{\varepsilon \to 0} | \Delta u_{\varepsilon} |_{L^2(0,T;H)}
	\le  |\Lambda_5|_{L^2(0,T)}, \\
	| \partial_{\boldsymbol{\nu}} u_{\kappa} |_{L^2(0,T;W_\Gamma^*)}
	\le \liminf_{\varepsilon \to 0} | \partial_{\boldsymbol{\nu}} u_{\varepsilon} |_{L^2(0,T;W_\Gamma^*)}
	\le  |\Lambda_6|_{L^2(0,T)}, \\
	| \xi_{\Gamma,\kappa} |_{L^2(0,T;W_\Gamma)} \le \liminf_{\varepsilon \to 0} 
	\bigl| \beta_{\Gamma,\varepsilon}( u_{\Gamma,\varepsilon}) \bigr|_{L^2(0,T;W_\Gamma)}
	\le |\Lambda_{11}|_{L^2(0,T)}.
	\label{el}
\end{gather}
Based on these uniform estimates, we can prove Theorem~\ref{convergence}.

\subsection{Proof of Theorem~2.1}
Thanks to the uniform estimates (\ref{e1})--(\ref{el}) \pier{and by compactness we deduce the following 
convergence properties: there exist a subsequence of $\kappa$ (not relabeled) and some limit functions 
$u, u_\Gamma, \mu, \xi, \xi_\Gamma$ such that the convergences \eqref{u1}--\eqref{xig1} and
\begin{align}
	\kappa u_{\Gamma,\kappa} \to 0 
	 \quad \text{strongly in } L^\infty(0,T;V_\Gamma)
	\label{ug2}
\end{align}
hold as $\kappa \searrow 0$.}
Now, by using (\ref{u1}) and (\ref{mu1}), \pier{taking the limit in the variational formulation~(\ref{vf}) yields \eqref{pier5}.}
About the equation \pier{in (\ref{chd2k}), we can pass to the limit directly 
with respect to} $\kappa$ and obtain 
\begin{equation}
	\mu = \tau \partial_t u -\Delta u 
	+ \xi + \pi (u)-g
	\quad \text{a.e.\ in }Q, 
	\label{equmu2}
\end{equation}
 \pier{\takeshi{whereas} for the inclusion in (\ref{chd2k}) it is easy to infer that
\begin{equation*}
	\xi \in \beta(u) \quad \text{a.e.\ in }Q,
\end{equation*}
by} (\ref{u1}), (\ref{beta1}) and the demi-closedness of $\beta$. \pier{Hence, also \eqref{pier6} is proved.}
The trace condition \pier{\eqref{pier7}, i.e.,} $u_{|_\Gamma}=u_\Gamma$ a.e.\ on $\Sigma$, is a consequence of 
(\ref{u1}) and (\ref{ug1}). The 
initial conditions \pier{\eqref{pier10} follow} from (\ref{u1}) and (\ref{ug1}) as well. 
It \pier{remains to pass to the limit in (\ref{chd5k}) in order to show \eqref{pier8} and \eqref{pier9}. 
Then, testing the equality  in (\ref{chd5k}) by an arbitrary $z_\Gamma \in V_\Gamma$, we have that}
\begin{align}
	& \int_\Gamma \partial_t u_{\Gamma,\kappa} z_\Gamma d\Gamma
	+ \langle \partial_{\boldsymbol{\nu}} u_{\kappa}, z_\Gamma \rangle_{W_\Gamma^*,W_\Gamma}
	+ \kappa \int_\Gamma \nabla _\Gamma u_{\Gamma,\kappa} \cdot \nabla_\Gamma z_\Gamma d\Gamma
	\nonumber \\
	& \quad {}	
	+ \langle \xi_{\Gamma,\kappa}, z_\Gamma \rangle_{W_\Gamma^*,W_\Gamma}
	+ \int_\Gamma \pi_\Gamma(u_{\Gamma,\kappa}) z_\Gamma d\Gamma
	= \int_\Gamma g_\Gamma z_\Gamma d\Gamma
	\label{vfg}
\end{align}
\pier{a.e.\ in $(0,T)$.  Taking the limit as $\kappa \searrow 0$, 
with the help of (\ref{u1}), \eqref{ug1}--(\ref{xig1}) and (\ref{ug2}) we obtain \eqref{pier8}
for all $z_\Gamma \in V_\Gamma$. As a further step, the density of $V_\Gamma$  in $W_\Gamma$ helps us to definitively show \eqref{pier8}.}
Let us point out that this implies that the equation 
\begin{equation}
	\partial_t u_{\Gamma}+
	\partial_{\boldsymbol{\nu}} u
	+ \xi_{\Gamma}
	+ \pi_\Gamma(u_{\Gamma})
	= g_\Gamma
	\quad \pier{\text{holds in }} L^2(0,T;W_\Gamma^*).
	\label{eqong}
\end{equation}
Next, we have to prove \eqref{pier9}.
Note here that from the inclusion in (\ref{chd5k}), we have \pier{that}
\begin{equation*}
	\xi_{\Gamma, \kappa} (t) \in \beta_{\Gamma,H_\Gamma} \bigl( u_{\Gamma,\kappa} (t)\bigr) 
	\quad \text{in }H_\Gamma , \  \text{for a.a.\ } t \in (0,T). 
\end{equation*}
Thus, due to  the convergences (\ref{ug1}) and (\ref{xig1}), it is enough to prove that 
\begin{equation}
	\limsup_{\kappa \searrow 0} \int_0^T 
	\bigl\langle \xi_{\Gamma, \kappa} (t), u_{\Gamma,\kappa}(t) \bigr\rangle_{W_\Gamma^*,W_\Gamma} dt
	\le \int_0^T 
	\bigl\langle \xi_{\Gamma} (t), u_{\Gamma}(t) \bigr\rangle_{W_\Gamma^*,W_\Gamma} dt
	\label{limsup}
\end{equation}
Indeed, taking $z_\Gamma:=u_{\Gamma,\kappa}$ in (\ref{vfg}) and integrating it with respect to time, we obtain 
\begin{align}
	&\int_0^T \bigl\langle \xi_{\Gamma,\kappa}(t), u_{\Gamma,\kappa}(t) \bigr\rangle_{W_\Gamma^*,W_\Gamma}dt \nonumber \\
	& = \int_0^T 
	\bigl( g_\Gamma (t) - \partial_t u_{\Gamma,\kappa}(t)-\pi_\Gamma \bigl(u_{\Gamma,\kappa}(t)\bigr), 
	u_{\Gamma,\kappa}(t) \bigr)_{\! H_\Gamma}dt  
	\nonumber 
	\\
	& \quad {}- \int_0^T \bigl\langle \partial_{\boldsymbol{\nu}} u_{\kappa}(t), u_{\Gamma,\kappa}(t)  
	\bigr\rangle_{W_\Gamma^*,W_\Gamma}dt 
	- \kappa \int_0^T\!\!\int_\Gamma \bigl| 
	\nabla _\Gamma u_{\Gamma,\kappa}(t) \bigr|^2 d\Gamma dt.
	\label{limsup2}
\end{align}
Next, we observe that an integration by part formula gives 
\begin{equation*}
	- \int_0^T \bigl\langle \partial_{\boldsymbol{\nu}} u_{\kappa}(t), u_{\Gamma,\kappa}(t)  
	\bigr\rangle_{W_\Gamma^*,W_\Gamma}dt 
	= -\int_0^T\!\!\int_\Omega \Delta u_{\kappa}(t) u_{\kappa}(t)dxdt 
	- \int_0^T\!\!\int_\Omega \bigl|\nabla u_{\kappa}(t) \bigr|^2 dx dt
\end{equation*}
\pier{and, in view of (\ref{u1}) and (\ref{Lap1}), taking the limit superior
yields} 
\begin{align*}
	& \limsup_{\kappa \searrow 0} 
	\left\{ -\int_0^T \bigl\langle \partial_{\boldsymbol{\nu}} u_{\kappa}(t), u_{\Gamma,\kappa}(t)  
	\bigr\rangle_{W_\Gamma^*,W_\Gamma}dt \right\} 
	\nonumber \\
	& 
	= - \int_0^T\!\!\int_\Omega \Delta u(t) u(t)dxdt 
	- \liminf _{\kappa \searrow 0}\int_0^T\!\!\int_\Omega \bigl|\nabla u_{\kappa}(t) \bigr|^2 dx dt 
	\nonumber \\
	& \le  - \int_0^T\!\!\int_\Omega \Delta u(t) u(t)dxdt 
	-\int_0^T\!\!\int_\Omega \bigl|\nabla u(t) \bigr|^2 dx dt	
	\pier{{} = - \int_0^T \bigl\langle \partial_{\boldsymbol{\nu}} u(t), 
	u_{\Gamma}(t)  \bigr\rangle_{W_\Gamma^*,W_\Gamma}dt}
\end{align*}
On the other hand, we \pier{see} that 
\begin{equation*}
	- \kappa \int_0^T\!\!\int_\Gamma \bigl| 
	\nabla _\Gamma u_{\Gamma,\kappa}(t) \bigr|^2 d\Gamma dt \le 0
\end{equation*} 
\pier{and} also point out that 
\begin{align*}
	& \lim_{\kappa \searrow 0} 
	\int_0^T 
	\bigl( g_\Gamma (t) - \partial_t u_{\Gamma,\kappa}(t)-\pi_\Gamma \bigl(u_{\Gamma,\kappa}(t)\bigr), 
	u_{\Gamma,\kappa}(t) \bigr)_{\! H_\Gamma}dt  
	\nonumber \\
	& = \int_0^T 
	\bigl( g_\Gamma (t) - \partial_t u_{\Gamma}(t)-\pi_\Gamma \bigl(u_{\Gamma}(t)\bigr), 
	u_{\Gamma,\kappa}(t) \bigr)_{\! H_\Gamma}dt.
\end{align*}
Collecting \pier{these remarks}, from (\ref{limsup2}) and (\ref{eqong}) we infer that 
\begin{align*}
	& \limsup_{\kappa \searrow 0} 
	\int_0^T \bigl\langle \xi_{\Gamma,\kappa}(t), u_{\Gamma,\kappa}(t) \bigr\rangle_{W_\Gamma^*,W_\Gamma}dt
	\nonumber \\
	& = \int_0^T 
	\bigl( g_\Gamma (t) - \partial_t u_{\Gamma}(t)-\pi_\Gamma \bigl(u_{\Gamma}(t)\bigr), 
	u_{\Gamma}(t) \bigr)_{\! H_\Gamma}dt  
	- \int_0^T \bigl\langle \partial_{\boldsymbol{\nu}} u(t), u_{\Gamma}(t)  
	\bigr\rangle_{W_\Gamma^*,W_\Gamma}dt 
	\nonumber \\
	& = \int_0^T \bigl\langle \xi_\Gamma (t), u_{\Gamma}(t)  
	\bigr\rangle_{W_\Gamma^*,W_\Gamma}dt 
\end{align*}
which concludes the proof of (\ref{limsup}), \pier{and consequently of (\ref{pier9}). We note that actually the 
convergences~\eqref{u1}--\eqref{xig1} 
do occur for the whole family as $k$ tends to $0$ since the limit quintuplet $(u,\mu, \xi, u_\Gamma, \xi_\Gamma )$  solves the 
problem \eqref{pier5}--\eqref{pier10} and this problem admits a unique solution (as from Theorem~\ref{contdep}). By this reamark
we} conclude the proof of Theorem~\ref{convergence}. \hfill $\Box$ 

\subsection{Proof of Theorem~2.2}

We can now prove the continuous dependence of \pier{the solutions with respect to} data. We \pier{recall \eqref{media} and underline} that both 
$u_{0,1}$ and $u_{0,2}$ satisfy (A4)  with the same value $m_0$. Now, we take the difference 
of \pier{(\ref{pier5}) written for $u_1, \mu_1 $ and $u_2, \mu_2$, then 
choose $z={\mathcal N}(u_1-u_2)$ (which is possible due to the mean value conservation property) and deduce that}
\begin{equation}
	\frac{1}{2} \frac{d}{dt} | u_1-u_2|_{V^*} ^2 
	+ ( \mu_1- \mu_2, u_1-u_2)_H=0
	\label{cont1}
\end{equation}
a.e.\ in $(0,T)$. 
Moreover, 
we take the difference of the \pier{equations in (\ref{pier6}) and 
test by $u_1-u_2$. We infer that} 
\begin{align}
	& ( \mu_1- \mu_2, u_1-u_2 )_H
	\nonumber \\
	& 
	= \frac{\tau}{2}\frac{d}{dt}
	|u_1-u_2|_H^2 
	+ \int_\Omega \bigl| \nabla (u_1-u_2)  \bigr|^2dx 
	- \bigl\langle \partial_{\boldsymbol{\nu}} (u_1-u_2), u_{\Gamma,1}-u_{\Gamma,2} \bigr \rangle_{W_\Gamma^*, W_\Gamma} 
	\nonumber \\	
	& \quad {}+ ( \xi_1-\xi_2, u_1-u_2 )_H 
	+ \bigl( \pi(u_1)-\pi(u_2),u_1-u_2\bigr)_H 
	- (g_1-g_2,u_1-u_2)_H
	\label{cont2}
\end{align}
a.e.\ in $(0,T)$. 
Now, with the help of \pier{(\ref{pier8}), we have that} 
\begin{align}
	& 
	- \bigl\langle \partial_{\boldsymbol{\nu}} (u_1-u_2), u_{\Gamma,1}-u_{\Gamma,2} 
	\bigr \rangle_{W_\Gamma^*, W_\Gamma} 
	\nonumber \\
	& = \frac{1}{2} \frac{d}{dt} |u_{\Gamma,1}-u_{\Gamma,2}|_{H_\Gamma}^2 
	+ \langle \xi_{\Gamma,1}-\xi_{\Gamma,2}, u_{\Gamma,1}-u_{\Gamma,2} \rangle_{W_\Gamma^*,W_\Gamma} 
	\nonumber \\
	& \quad {}	
	+ \bigl( \pi_\Gamma(u_{\Gamma,1})-\pi_\Gamma(u_{\Gamma,2}),u_{\Gamma,1}-u_{\Gamma,2} \bigr)_{\! H_\Gamma} 
	- (g_{\Gamma,1}-g_{\Gamma,2},u_{\Gamma,1}-u_{\Gamma,2})_{H_\Gamma}
	\label{cont3}
\end{align}
a.e.\ in $(0,T)$. 
Combining (\ref{cont1})--(\ref{cont3}), adding 
$|u_1-u_2|_H^2$, 
applying the Young inequality, and 
using  the monotonicity of $\beta$ and $\beta_\Gamma$,
we obtain 
\begin{align}
	& \frac{1}{2} \frac{d}{dt} | u_1-u_2|_{V^*} ^2 
	+ \frac{\tau}{2}\frac{d}{dt}
	|u_1-u_2|_H^2 
	+ | u_1-u_2 |_V^2
	+\frac{1}{2} \frac{d}{dt} |u_{\Gamma,1}-u_{\Gamma,2}|_{H_\Gamma}^2 
	\nonumber \\
	& \le  |u_1-u_2|_H^2  + L |u_1-u_2|_H^2 
	+  |u_1-u_2|_H^2+ \frac{1}{4} |g_1-g_2|_H^2
	\nonumber \\
	& \quad {}
	+ L_\Gamma |u_{\Gamma,1}-u_{\Gamma,2}|_{H_\Gamma}^2 
	+ |u_{\Gamma, 1}-u_{\Gamma,2}|_{H_\Gamma}^2 
	+ \frac{1}{4} |g_{\Gamma,1}-g_{\Gamma,2}|_{H_\Gamma}^2
	\nonumber \\
	& \le (2+L ) \bigl( \delta |u_1-u_2|_V^2 
	+ C(\delta)  |u_1-u_2|_{V^*}^2 \bigr) 
	+  (1+L_\Gamma) |u_{\Gamma,1}-u_{\Gamma,2}|_{H_\Gamma}^2 
	\nonumber \\
	& \quad {}
	+ \frac{1}{4} |g_1-g_2|_H^2
	+ \frac{1}{4} |g_{\Gamma,1}-g_{\Gamma,2}|_{H_\Gamma}^2
	\label{cont4}
\end{align}
a.e.\ in $(0,T)$, 
where we used the compactness inequality \pier{(see, e.g., \cite[Lemme~5.1, p.~58]
{Lions} or \cite[Theorem~16.4, p.~102]{LM72})}
with the parameter $\delta >0$ and 
some constant $C(\delta)$. Now we take $\delta =1/2(2+L)$, then 
\pier{integrate from $0$ to $t\in [0,T]$ using the initial conditions and finally apply the Gronwall lemma. We} deduce that
\begin{align*}
	& \bigl|u_1(t)-u_2(t) \bigr|_{V^*}^2 
	+ \tau \bigl|u_1(t)-u_2(t) \bigr|_{H}^2
	\nonumber \\
	&\quad \pier{{}+ \int_0^t \takeshi{\bigl|} (u_1-u_2)(s) \takeshi{\bigr |}_V^2ds}
	+ \bigl|u_{\Gamma,1}(t)-u_{\Gamma,2}(t) \bigr|_{H_\Gamma}^2
	\nonumber \\
	& \le C_8 \Bigl\{ 
	 |u_{0,1}-u_{0,2}|_{V^*}^2 
	+ \tau |u_{0,1}-u_{0,2}|_{H}^2
	+ |u_{0\Gamma,1}-u_{0\Gamma,2}|_{H_\Gamma}^2
	\nonumber \\
	& \quad {} + 
	|g_1-g_2|_{L^2(0,T;H)}^2 
	+ 
	|g_{\Gamma,1}-g_{\Gamma,2}|_{L^2(0,T;H_\Gamma)}^2 
	\Bigr\}
\end{align*}
for all $t \in [0,T]$, where $C_8$ is a positive constant depending on $L,L_\Gamma$, and $T$. 
Finally, \pier{by the last inequality we easily arrive 
at the desired estimate~(\ref{dep}).} \hfill $\Box$

\section{Improvement of the convergence-existence theorem}
\setcounter{equation}{0}

The aim of this section is an improvement \pier{of Theorem~\ref{convergence}} in the case when the two graphs $\beta$ and $\beta_\Gamma$ 
grow at the same kind of rate, more precisely: 

\begin{enumerate}
\item[(A2)$^\prime$]  \pier{$D (\beta_\Gamma) = D(\beta)$ and there exist \pier{two constants $\varrho \geq 1$ and $c_0>0$} such that 
\begin{equation}
	\frac 1 \varrho \bigl| \beta ^\circ_\Gamma (r) \bigr|-c_0 
	\le \bigl| \beta^\circ(r) \bigr| \le \varrho \bigl| \beta ^\circ_\Gamma (r) \bigr|+c_0 \quad 
	\text{for all } r \in D(\beta)\equiv D(\beta_\Gamma).
	\label{cccond2}
\end{equation}}
\end{enumerate}
\vskip-5pt

Under this assumption, the same \pier{inequalities as \eqref{cccond2} hold for the {Y}osida  regularizations $\beta_\varepsilon$ and $\beta_{\Gamma, \varepsilon}$, that is, 
\begin{gather}
	\frac 1 \varrho  \bigl| \beta_{\Gamma, \varepsilon} (r) \bigr|-c_0 
	\le
	\bigl |\beta_\varepsilon (r)\bigr | 
	\le \varrho \bigl |\beta _{\Gamma,\varepsilon } (r)\bigr |+c_0
	\quad 
	\text{for all } r \in \mathbb{R},
	\label{cccond2e}
\end{gather} 
for all $\varepsilon \in (0,1]$, with the same constants 
$\varrho$ and $c_0$ (see Appendix).}
\smallskip

\pier{Under the assumption (A2)$^\prime$, the conclusion of Theorem~\ref{convergence} can be improved. In particular, the solution $(u,\mu, \xi, u_\Gamma, \xi_\Gamma )$ to \eqref{pier5}--\eqref{pier10} is more regular.}

\begin{theorem}
\label{regularity}
\pier{Let $\tau \ge 0$ and assume that {\rm (A1), (A2)$^\prime$, (A3)--(A5)} hold. For all $\kappa\in (0,1]$
let 
$(u_\kappa ,\mu_\kappa , \xi_\kappa , u_{\Gamma, \kappa }, \xi_{\Gamma, \kappa } )$ denote the solution to 
\eqref{chd1}--\eqref{chd6}  defined by Proposition~\ref{prop}. Then
the convergences \eqref{u1}--\eqref{nu1} and 
\begin{align}
	\xi_{\Gamma,\kappa} \to \xi_\Gamma
	&  \quad \text{weakly in } L^2 (0,T;H_\Gamma) 
	\label{pier15}
\end{align}
hold as $\kappa \searrow 0$, where $(u,\mu, \xi, u_\Gamma, \xi_\Gamma )$ 
is the solution to \eqref{pier5}--\eqref{pier10} and satisfies 
\begin{align}
	& u \in H^1(0,T;V^*)\cap C \takeshi{ \bigl (} [0,T];H \takeshi{ \bigr )} \cap L^\infty (0,T;V) \cap L^2 \bigl(0,T;H^{3/2}(\Omega) \bigr), \\
	& \tau u \in H^1(0,T;H), \quad 
	 \mu \in L^2(0,T;V), \quad \xi \in L^2(0,T;H), \\
	& u_\Gamma \in H^1(0,T;H_\Gamma )\cap \takeshi{C} \bigl([0,T]; W_\Gamma \bigr) \cap L^2 \bigl( 0,T;V_\Gamma \bigr), 
	\label{pier18}\\
	& \partial_{\boldsymbol{\nu}} u \in L^2(0,T;H_\Gamma), \quad
	 \xi_\Gamma \in L^2(0,T;H_\Gamma). 
\end{align}
Moreover, we have that
\begin{gather}
	\partial_t u_{\Gamma}+
	\partial_{\boldsymbol{\nu}} u
	+ \xi_{\Gamma}
	+ \pi_\Gamma(u_{\Gamma})
	= g_\Gamma, 
	\quad \xi_\Gamma \in \beta_\Gamma(u_\Gamma) \quad \text{a.e.\ on } \Sigma
	\label{pier16}
\end{gather}
as improvement of \eqref{pier8} and \eqref{pier9}.}
\end{theorem}

\begin{proof} \pier{We refer to the estimate~(\ref{37-5}) in Lemma~3.7
and show that (A2)$^\prime$ and consequently \eqref{cccond2e} allow us to 
produce a better estimate. Indeed, thanks to (\ref{cccond2e}) and \eqref{35} we can obtain 
\begin{align*}
	\bigl| \beta_{\Gamma,\varepsilon} \bigl( u_{\Gamma,\varepsilon}(t) \bigr) \bigr|_{H_\Gamma}^2 
	& \le \int_\Gamma \left( \varrho \bigl| \beta_{\varepsilon} 
	\bigl(u_{\Gamma,\varepsilon} (t) \bigr)  \bigr|+ c_0 \right)^2 d\Gamma 
	\nonumber \\
	& \le 2 \varrho^2 
	\bigl| \beta_{\varepsilon} \bigl( u_{\Gamma,\varepsilon}(t) \bigr) \bigr|_{H_\Gamma}^2 
	+  2 c_0^{\,2}|\Gamma| 
	\nonumber \\
	& \le C_9\bigl( 1+\Lambda_4(t)^2 \bigr)
\end{align*}
for a.a.\ $t \in (0,T)$, where the right-hand side is bounded in $L^1(0,T)$. Thus, the estimate (\ref{37-5}) in Lemma~3.7 is replaced by the key 
estimate
\begin{equation*}
	\bigl| \beta_{\Gamma,\varepsilon} \bigl( u_{\Gamma,\varepsilon}(t) \bigr) \bigr|_{H_\Gamma}
	\le \Lambda_{12}(t)
\end{equation*}
for a.a.\ $t \in (0,T)$, where $\Lambda_{12}$ is bounded in $L^2(0,T)$. 
Hence, the convergence in (\ref{xig1}) can be improved to \eqref{pier15}.}
Using the demi-closedness of the maximal monotone operator \pier{induced by
$\beta_\Gamma$ on} $L^2(0,T;H_\Gamma)$, in the light of (\ref{ug1}) and \pier{(\ref{pier15})} we can show that 
\begin{equation*}
	\xi_\Gamma \in \beta_\Gamma (u_\Gamma) 
	\quad \text{a.e.\ on } \Sigma,
\end{equation*}
which improves \pier{(\ref{pier9})}. 
Moreover, \pier{from a comparison of the terms in (\ref{eqong}) it turns out that
\begin{equation}
	\partial_{\boldsymbol{\nu}} u \in L^2(0,T;H_\Gamma), \label{pier17}
\end{equation}
thus the condition (\ref{eqong}) on the boundary actually holds in 
$L^2(0,T;H_\Gamma)$.
The additional information \eqref{pier17}, along with 
\begin{align*}
	u \in L^2(0,T;V) ,  \quad  \Delta u \in L^2(0,T;H),
\end{align*}
implies that $u \in L^2(0,T;H^{3/2}(\Omega))$. Indeed, one can apply the elliptic regularity theorem 
\cite[Theorem~3.2, p.~1.79]{BG87} to $u(t) \in V$ solving
\begin{equation*}
\begin{cases}
	-\Delta u (t)= \tilde g  (t)\quad \text{a.e.\ in } \Omega, 
	\\[1mm]
	\partial_{\boldsymbol{\nu}} u (t) = \tilde{g}_\Gamma (t) \quad \text{a.e.\ on } \Gamma
\end{cases}
\end{equation*}
for a.a. $t\in (0,T)$, where 
\begin{align*}
&\tilde g=\mu-\tau \partial_t u -\xi -\pi(u)+g \in L^2(0,T;H), \\
&\tilde{g}_\Gamma= g_\Gamma - \partial_t u_\Gamma -\xi_\Gamma -\pi_\Gamma(u_\Gamma)  
\in L^2(0,T;H_\Gamma).
\end{align*}
As a consequence of  $u \in L^2(0,T;H^{3/2}(\Omega))$, the trace theory enables us deduce that  
$u_\Gamma =u_{|_\Gamma} \in L^2(0,T;V_\Gamma)$, whence \eqref{pier18} follows by interpolation.}
\end{proof}

\section*{Appendix}
\renewcommand{\theequation}{a.\arabic{equation}}
\setcounter{equation}{0}

We use the same setting as in the previous sections.

\begin{lemmaa}
Assume {\rm (A2)}. Then {\rm (\ref{ccconde})} holds, \pier{that is,}
\begin{equation*}
	\bigl |\beta_\varepsilon (r)\bigr | 
	\le \varrho \bigl |\beta _{\Gamma,\varepsilon } (r)\bigr |+c_0
	\quad 
	\text{for all } r \in \mathbb{R},
\end{equation*} 
for all $\varepsilon \in (0,1]$ with the same constants \pier{$\varrho \geq 1 $ and $c_0 >0 $.}
\end{lemmaa}
\begin{proof} Thanks to \cite[Lemma~4.4]{CC13}, \pier{it is known that} 
\begin{equation*}
	\bigl |\beta_\varepsilon (r)\bigr | 
	\le \varrho \bigl |\beta _{\Gamma,\varepsilon \varrho } (r)\bigr |+c_0
	\quad 
	\text{for all } r \in \mathbb{R},
\end{equation*} 
where \pier{$\beta _{\Gamma,\varepsilon \varrho }$ denotes the Yosida approximation of $\beta _{\Gamma }$
with parameter $\varepsilon \varrho$, i.e.,}
\begin{equation}
	\beta _{\Gamma, \varepsilon \varrho } (r)
	:= \frac{1}{\varepsilon \varrho} \bigl( r-(I+\varepsilon \varrho \beta _\Gamma )^{-1} (r) \bigr ).
	\label{cc}
\end{equation}
\pier{Then, recalling that $\varrho \ge 1$,  we may invoke the fundamental property 
\cite[Proposition~2.6, p.~28]{Bre73} of {Y}osida approximations, which implies that} 
\begin{equation*}
	\bigl| \beta_{\Gamma, \varepsilon \varrho} (r) \bigr| \le |\beta_{\Gamma, \varepsilon} (r) \bigr|
	\quad \text{for all } r \in \mathbb{R},
\end{equation*}
because $\varepsilon \le \varepsilon \varrho$. Thus we get the conclusion. 
\end{proof}

\begin{lemmaa} 
Assume {\rm (A2)$^\prime$}. Then {\rm (\ref{cccond2e})} holds, \pier{that is, 
\begin{equation*}
	\frac 1 \varrho  \bigl| \beta_{\Gamma, \varepsilon} (r) \bigr|-c_0 
	\le
	\bigl |\beta_\varepsilon (r)\bigr | 
	\le \varrho \bigl |\beta _{\Gamma,\varepsilon } (r)\bigr |+c_0
	\quad 
	\text{for all } r \in \mathbb{R},
\end{equation*} 
for all $\varepsilon \in (0,1]$ with the same constants $\varrho \geq 1 $ and $c_0 >0$.}
\end{lemmaa}
\begin{proof} 
\pier{In view of Lemma~A.1, is enough to prove that 
\begin{equation*}
	\frac 1 \varrho  \bigl| \beta_{\Gamma, \varepsilon} (r) \bigr|-c_0 
	\le
	\bigl |\beta_\varepsilon (r)\bigr | 
	\quad 
	\text{for all } r \in \mathbb{R},
\end{equation*} 
which is the same as  
\begin{equation*}
	\bigl |\beta _{\Gamma,\varepsilon } (r)\bigr | 
	\le \varrho\bigl |\beta_\varepsilon (r)\bigr | 
	+
	 \varrho \,c_0
	\quad 
	\text{for all } r \in \mathbb{R}.
\end{equation*} 
But this follows immediately from  Lemma~A.1 again.}
\end{proof}

\begin{remarka}
\pier{\rm Comparing to previous works (see, e.g., \cite{CF15b, CF15, CGS14}\takeshi{)} in which 
the same kind of property~{\rm (\ref{cccond})}
was assumed for the two maximal monot\takeshi{one} graphs, the parameter of the \pier{{Y}osida}  regularizations 
\pier{is here} 
the same for both graphs (see also \cite[Section~2]{CGNS17}) . 
Instead, in the approach devised in {\rm \cite[Lemma~4.4]{CC13}} exactly the approximation  
$\beta _{\Gamma, \varepsilon \varrho}$ defined by (\ref{cc}) was introduced and used for $\beta _{\Gamma}$ .} 
\end{remarka}


\section*{Acknowledgments}
The \pier{authors warmly thank Professor Ken Shirakawa for his valuable advice about~Lemmas~A.1 and~A.2.
P.~Colli acknowledges support 
from the Italian Ministry of Education, 
University and Research~(MIUR): Dipartimenti di Eccellenza Program (2018--2022) 
-- Dept.~of Mathematics ``F.~Casorati'', University of Pavia, and from 
the GNAMPA~(Gruppo Nazionale per l'Analisi Matematica, 
la Probabilit\`a e le loro Applicazioni) of INdAM (Isti\-tuto 
Nazionale di Alta Matematica).} 
T.~Fukao acknowledges the support
from the JSPS KAKENHI Grant-in-Aid for Scientific Research(C), Japan, Grant Number 17K05321 and 
from the Grant Program of The Sumitomo Foundation, Grant Number 190367.

\end{document}